\documentclass[a4paper, 12pt]{article}

\usepackage[utf8]{inputenc}
\usepackage[T1]{fontenc}
\usepackage[a4paper, margin=2.5cm]{geometry}
\usepackage{needspace}

\usepackage{libertine} % text font
\usepackage{inconsolata} % texttt font

\usepackage{amsmath, amsthm, amssymb}
\usepackage{graphicx}
\usepackage{enumerate}
\usepackage{enumitem}
\usepackage{authblk}
\usepackage[textsize=small, textwidth=2cm, color=yellow]{todonotes}
\usepackage{thmtools}
\usepackage{thm-restate}
\usepackage[colorlinks=true, citecolor=red]{hyperref}
\usepackage{float}

\usepackage{tikz}
\usepackage{pgfplots}
\usetikzlibrary{calc}
\usepackage{xspace}

\usepackage[capitalize]{cleveref}
\crefname{section}{Section}{Sections}

\renewenvironment{abstract}
{\small\vspace{-1em}
\begin{center}
\bfseries\abstractname\vspace{-.5em}\vspace{0pt}
\end{center}
\list{}{
\setlength{\leftmargin}{0.6in}%
\setlength{\rightmargin}{\leftmargin}}%
\item\relax}
{\endlist}

\declaretheorem[name=Theorem,     refname={Theorem,Theorems},         numberwithin=section]{theorem}
\declaretheorem[name=Lemma,       refname={Lemma,Lemmas},             sibling=theorem]{lemma}

\declaretheorem[name=Corollary,   refname={Corollary,Corollaries},    sibling=theorem]{corollary}
\declaretheorem[name=Conjecture,  refname={Conjecture,Conjectures},   sibling=theorem]{conjecture}

\declaretheorem[name=Claim,       refname={Claim,Claims},             sibling=theorem]{claim}

\def\cqedsymbol{\ifmmode$\lrcorner$\else{\unskip\nobreak\hfil
\penalty50\hskip1em\null\nobreak\hfil$\lrcorner$
\parfillskip=0pt\finalhyphendemerits=0\endgraf}\fi} 
\newcommand{\cqed}{\renewcommand{\qed}{\cqedsymbol}}

\def\eps{\varepsilon}
\makeatletter
\newcommand{\leqnomode}{\tagsleft@true}
\newcommand{\reqnomode}{\tagsleft@false}
\makeatother

\newcommand\eqdef{\overset{\text{\tiny{def}}}{=}} 

\newcommand{\card}[1]{\left|#1\right|}
\newcommand{\fvs}{\mathrm{fvs}\xspace}
\newcommand{\tw}{\mathrm{tw}\xspace}
\newcommand{\ocp}{\mathrm{ocp}\xspace}
\newcommand{\cp}{\mathrm{cp}\xspace}
\newcommand{\icp}{\mathrm{icp}\xspace}
\newcommand{\iocp}{\mathrm{iocp}\xspace}

\newcommand{\mis}{\textsc{Maximum Independent Set}\xspace}
\newcommand{\smis}{\textsc{MIS}\xspace}

\let\le\leqslant
\let\ge\geqslant

\let\geq\geqslant

\title{Sparse graphs with bounded induced cycle packing number have logarithmic treewidth \thanks{This is the full version of a paper which was presented at the 2023 Annual ACM-SIAM Symposium on Discrete Algorithms (SODA 2023) \cite{soda}.}}
\author[1]{Marthe Bonamy\thanks{Supported by ANR project DISTANCIA (Metric Graph Theory, ANR-17-CE40-0015).}}
\author[2]{\'{E}douard Bonnet\thanks{É. B., H. D., L. E., C. G., and S. T. were supported by the ANR projects TWIN-WIDTH (ANR-21-CE48-0014-01) and Digraphs (ANR-19-CE48-0013-01).}}
\author[2]{Hugues D\'{e}pr\'{e}s}
\author[3]{Louis Esperet\thanks{Partially supported by LabEx
  PERSYVAL-lab (ANR-11-LABX-0025).}}
\author[2]{Colin Geniet}
\author[4]{Claire Hilaire\thanks{Partially supported by Slovenian Research and Innovation Agency (research project J1-4008).}}
\author[2]{St\'{e}phan Thomass\'{e}}
\author[5]{Alexandra Wesolek \thanks{Supported by the Vanier Canada Graduate Scholarships program and by the Deutsche Forschungsgemeinschaft (DFG, German Research Foundation) under Germany's Excellence Strategy – The Berlin Mathematics Research Center MATH+ (EXC-2046/1, project ID: 390685689).}}

\affil[1]{CNRS, LaBRI, Université de Bordeaux, Bordeaux, France.}
\affil[2]{Univ Lyon, CNRS, ENS de Lyon, Université Claude Bernard Lyon 1, LIP UMR5668, France.}
\affil[3]{CNRS, G-SCOP, Université Grenoble Alpes, Grenoble, France.}
\affil[4]{FAMNIT, University of Primorska, Slovenia.}
\affil[5]{Technische Universität Berlin, Berlin, Germany.}

\date{\today}

\begin{document}

\maketitle

\begin{abstract}
A graph is $\mathcal{O}_k$-free if it does not contain $k$ pairwise vertex-disjoint and non-adjacent cycles. We prove that "sparse" (here, not containing large complete bipartite graphs as subgraphs) $\mathcal{O}_k$-free graphs have treewidth (even, feedback vertex set number) at most logarithmic in the number of vertices. This is optimal, as there is an infinite family of $\mathcal{O}_2$-free graphs without $K_{2,3}$ as a subgraph and whose treewidth is (at least) logarithmic.

Using our result, we show that \textsc{Maximum Independent Set} and \textsc{3-Coloring} in $\mathcal{O}_k$-free graphs can be solved in quasi-polynomial time.
Other consequences include that most of the central NP-complete problems (such as \textsc{Maximum Independent Set}, \textsc{Minimum Vertex Cover}, \textsc{Minimum Dominating Set}, \textsc{Minimum Coloring}) can be solved in polynomial time in sparse $\mathcal{O}_k$-free graphs, and that deciding the $\mathcal{O}_k$-freeness of sparse graphs is polynomial time solvable. 
\end{abstract}

% edition
%\overfullrule=50pt % to spot overfull hbox

\section{Introduction}\label{sec:intro}

Two vertex-disjoint subgraphs $H$ and $H'$ in a graph $G$ are \emph{independent} if there is no edge between $H$ and $H'$ in $G$. 
\emph{Independent cycles} %\ed{for some people independent cycles are vertex-disjoint cycles (see for instance "On independent cycles and edges in graphs" by Thomas Andreae). This meaning seems to be losing traction, so it's probably fine to overwrite it. \textcolor{orange}{MB: Agree it's fine}} 
are simply vertex-disjoint cycles that are pairwise independent. 
Let $\mathcal{O}_k$ denote the family of all graphs consisting of the disjoint union of $k$ cycles. We say that a graph is \emph{$\mathcal{O}_k$-free} if it does not contain any graph of $\mathcal{O}_k$ as an induced subgraph. Equivalently, a graph $G$ is $\mathcal{O}_k$-free
if it does not contain $k$ independent induced cycles, 
%or (equivalently) if  $\icp(G) < k$, 
or (equivalently), if $G$ does not contain $k$ independent  cycles.
These graphs can equivalently be defined in terms of forbidden induced subdivisions. 
Letting $T_k$ be the disjoint union of $k$ triangles, a graph is $\mathcal{O}_k$-free if and only if it does not contain an induced subdivision of $T_k$.

\medskip

A \emph{feedback vertex set} is a set of vertices whose removal yields a forest.
Our main technical contribution is the following.
\begin{restatable}{theorem}{mainthm}
  \label{thm:mainthm}
  Every $\mathcal{O}_k$-free graph on $n$ vertices that does not contain $K_{t,t}$ as a~subgraph has a~feedback vertex set of size $O_{t,k}(\log n)$.
\end{restatable}
Since a graph with a feedback vertex set of size $k$ has treewidth at most $k+1$,
this implies a corresponding result on treewidth.

\begin{restatable}{corollary}{maintw}
  \label{cor:main-treewidth}
  Every $\mathcal{O}_k$-free graph on $n$ vertices that does not contain $K_{t,t}$ as a subgraph has treewidth $O_{t,k}(\log n)$.
\end{restatable}

\Cref{cor:main-treewidth} implies that a number of fundamental problems, such as \textsc{Maximum Independent Set}, \textsc{Minimum Vertex Cover}, \textsc{Minimum Dominating Set}, \textsc{Minimum Colo}\-\textsc{ring}, can be solved in polynomial time in "sparse" $\mathcal{O}_k$-free graphs. Before we elaborate on the algorithmic consequences of our results, we mention that our work is related to an ongoing project  devoted to unraveling an \emph{induced} version of the grid minor theorem of Robertson and Seymour~\cite{RS86}. 
This theorem implies that every graph not containing a subdivision of a $k \times k$ wall as a subgraph has treewidth at most $f(k)$, for some function $f$. 
This result had a deep impact in algorithmic graph theory since many natural problems are tractable in graphs of bounded treewidth. 

Now, what are the forbidden \emph{induced} subgraphs in graphs of small treewidth?
It is clear that large cliques, complete bipartite graphs, subdivided walls, and line graphs of subdivided walls shall be excluded.
It was actually suggested that in graphs with no $K_{t,t}$ subgraphs, the absence of induced subdivisions of large walls and their line graphs might imply bounded treewidth, but counterexamples were found~\cite{Sintiari21,Dav22,Tro22}. 
However, Korhonen recently showed that this absence suffices within bounded-degree graphs~\cite{Korhonen22}.
%On the positive side, it was proven in \cite{ACD21} that for graphs of bounded maximum degree, forbidding thetas and pyramids in addition to line graphs of walls as induced subgraphs implies bounded treewidth. 
Abrishami et al.~\cite{Abrishami22} proved that a vertex with at least two neighbors on a hole (i.e., an induced cycle of length at least four) is also necessary in a counterexample. 
Echoing our main result, it was proven that (triangle,theta)-free graphs have logarithmic treewidth~\cite{ACH22}, where a \emph{theta} is made of three paths each on at least two edges between the same pair of vertices. The interested reader is referred to \cite{ACHS23a,ACHS23b} for more recent development on the ongoing project.

As we shall see, the class of $\mathcal{O}_2$-free graphs that do not contain $K_{3,3}$ as a subgraph
has unbounded treewidth. 
Since these graphs do not contain as an induced subgraph a subdivision of a large wall or its line graph, they constitute yet another family of counterexamples. 
%To our knowledge, this also yields the first example of hereditary classes with unbounded treewidth in which every graph has a feedback vertex set of size logarithmic in its number of vertices. 

We leave as an open question whether $\mathcal{O}_k$-free graphs that do not contain $K_{t,t}$ as a subgraph have bounded \emph{twin-width}, that is, if there is a function $f: \mathbb N \times \mathbb N \to \mathbb N$ such that their twin-width is at most~$f(t,k)$, and refer the reader to~\cite{twin-width1} for a definition of twin-width.

\subsection*{Algorithmic motivations and consequences}

A natural approach to tackle NP-hard graph problems is to consider them on restricted classes.
A simple example is the case of forests, that is, graphs \emph{without cycles}, on which most hard problems become tractable.
The celebrated Courcelle's theorem~\cite{Courcelle90} generalizes that phenomenon to graphs of bounded treewidth and problems expressible in monadic second-order logic.

For the particular yet central \mis (\smis, for short), the mere absence of \emph{odd} cycles makes the problem solvable in polynomial time.
Denoting by $\ocp(G)$ (for odd cycle packing) the maximum cardinality of a collection of vertex-disjoint odd cycles in~$G$, the classical result that \smis is polytime solvable in bipartite graphs corresponds to the $\ocp(G)=0$ case.
Artmann et al.~\cite{Artmann17} extended the tractability of~\smis to graphs $G$ satisfying $\ocp(G) \leqslant 1$.
One could think that such graphs are close to being bipartite, in the sense that the removal of a few vertices destroys all odd cycles.
This is not necessarily true: Adding to an $n \times n$ grid the edges between $(1,i)$ and $(n,n+1-i)$, for every $i=1, \ldots, n$, yields a graph $G$ with $\ocp(G)=1$ such that no removal of less than $n$ vertices make $G$ bipartite; also see the \emph{Escher wall} in~\cite{Reed99}.   

It was believed that Artmann et al.'s result could even be lifted to graphs with bounded odd cycle packing number.
Conforti et al.~\cite{Conforti20} proved it on graphs further assumed to have bounded genus, and Fiorini et al. \cite{Fiorini21} confirmed the general conjecture for graphs with bounded odd cycle packing number.
A polynomial time approximation scheme (PTAS), due to Bock et al.~\cite{Bock14}, was known for~\smis in the (much) more general case of $n$-vertex graphs $G$ such that $\ocp(G)=o(n/\log n)$. 

Similarly let us denote by $\cp(G), \icp(G), \iocp(G)$ the maximum cardinality of a collection of vertex-disjoint cycles in $G$ that are unconstrained, independent, and independent and of odd length, respectively (for cycle packing, induced cycle packing, and induced odd cycle packing).
The Erd\H{o}s-P\'osa theorem~\cite{EP65} states that graphs $G$ with $\cp(G)=k$ admit a feedback vertex set (i.e., a subset of vertices whose removal yields a forest) of size $O(k \log k)$, hence have treewidth $O(k \log k)$.
Thus, graphs with bounded cycle packing number allow polynomial time algorithms for a wide range of problems.

%The main paradigm to address hard computational problems in  graphs is to limit the scope to particular classes of structures. For instances most problems become easy on trees (acyclic graphs) and some particular ones (like maximum independent set) are easier on bipartite graphs (graphs without odd cycles). A natural way of adding restrictions to graph classes is indeed to limit the occurrence of cycles. 

%For instance the class of graphs $G$ which do not contain two disjoint cycles is also tractable since in this case, one can remove at most three vertices to make $G$ acyclic (such a set of vertices intersecting all cycles is called a \emph{feedback vertex set}). More generally, a classical result from Erd\H{o}s and P\'osa \cite{EP65} shows that the size of the feedback vertex set is at most some function of the maximum number of disjoint cycles. Therefore, the class of graphs without $k$ disjoint cycles has treewidth at most $f(k)$, hence is computationally tractable.

However, graphs with bounded feedback vertex set are very restricted.
This is a motivation to consider the larger classes for which solely the induced variants $\icp$ and $\iocp$ are bounded. 
Graphs with $\iocp \leqslant 1$ have their significance since they contain all the complements of disk graphs~\cite{BonnetG0RS18} and all the complements of unit ball graphs~\cite{BonamyBBCT18}.  
Concretely, the existence of a polynomial time algorithm for \smis on graphs with $\iocp \leqslant 1$ ---an intriguing open question--- would solve the long-standing open problems of whether \textsc{Maximum Clique} is in~P for disk graphs and unit ball graphs.
Currently only efficient PTASes are known~\cite{BBB21}, even when only assuming that $\iocp \leqslant 1$ and that the solution size is a positive fraction of the total number of vertices~\cite{DP21}. 
Let us mention that recognizing the class of graphs $G$ satisfying $\iocp(G) \leqslant 1$ is NP-complete~\cite{Golovach12}.

We have seen that graphs with bounded $\cp, \ocp, \iocp$ have been studied in close connection with solving \smis (or a broader class of problems), respectively forming the Erd\H{o}s-P\'osa theory, establishing a far-reaching generalization of total unimodularity, and improving the approximation algorithms for~\textsc{Maximum Clique} on some geometric intersection graph classes.
Relatively less attention has been given to $\icp$. As a graph $G$ satisfies $\icp(G) < k$ if and only if it is $\mathcal{O}_k$-free, our results (and their algorithmic consequences) are precisely about graphs with bounded induced cycle packing, with a particular focus on the sparse case.

%and one can wonder if we can add less restrictions to cycles than merely bounding the way they pack. %Two major directions can be followed: First one can limit the number of disjoint odd cycles and try to solve the maximum independent set problem. This is by no mean an easy task since there are graphs without two disjoint odd cycles in which the removal of any fixed number of vertices does not leave the graph bipartite. 

%Second, one can still authorize disjoint cycles, but forbid that two cycles are "far apart", that is every pair of cycles either intersect or have an edge between them. These two approaches capture some "real life" instances of problems, for instance the graph $G^3$ obtained by linking pairs of points in the 3D space when their distance is at least 1 does not have two odd cycles which are far apart \cite{BBB21} (see also \cite{DP21} for algorithmic aspects of graphs with no $k$ odd cycles that are far apart). Note that $G^3$ enjoys a combination of both approaches, and that being able to solve the maximum independent set in this case would allow to compute a maximum number of pairwise intersecting unit balls in $\mathbb R^3$ out of any collection given as input.

\medskip

So, what can be said about the complexity of classical optimization problems for $\mathcal{O}_k$-free graphs?
Even the class of $\mathcal{O}_2$-free graphs is rather complex.
Observe indeed that complements of graphs without $K_{3,3}$ subgraph are $\mathcal{O}_2$-free. 
As~\smis remains NP-hard in graphs with girth at least 5 (hence without $K_{3,3}$ subgraph)~\cite{Alekseev82}, \textsc{Maximum Clique} is NP-hard in $\mathcal{O}_2$-free graphs. 
Nonetheless \smis could be tractable in $\mathcal{O}_k$-free graphs, as is the case in graphs of bounded $\ocp$:

\begin{conjecture}\label{conj:mis}
\mis is solvable in polynomial time in $\mathcal{O}_k$-free graphs.
\end{conjecture}

As far as we can tell, \smis could even be tractable in graphs with bounded $\iocp$.
This would be a surprising and formidable generalization of~\cref{conj:mis} and of the same result for bounded $\ocp$~\cite{Fiorini21}.

 %Let us observe that quasipolynomial-time approximation schemes (QPTASes) were obtained in the setting of~\cref{conj:mis}, even in the more general case of $\iocp < k$: in deterministic time $2^{O(k \log^6 n)}$~\cite{BBB21}, and in randomized time $2^{O(k \log^2 n)}$~\cite{DP21}.
 %\smis also admits an exact $2^{\Tilde{O}(n^{2/3})}$-time algorithm for $n$-vertex graphs $G$ such that $\iocp(G) \leqslant 1$~\cite{BBB21} (a superclass of $O_2$-free graphs).   
 %As a consequence of~\cref{cor:main-treewidth}, we give for \smis a~simple deterministic QPTAS in time $2^{O(\log^2 n)}$ in $O_2$-free graphs, and an exact subexponential algorithm in time $2^{k \cdot O(\sqrt n \log n)}$ in $\mathcal{O}_k$-free graphs.
 
 We note that  \Cref{cor:main-treewidth} implies~\cref{conj:mis} in the sparse case. We come short of proving~\cref{conj:mis} in general, but not by much. 
 We obtain a quasi\-polynomial time algorithm for \smis in general $\mathcal{O}_k$-free graphs, excluding that this problem is NP-complete without any complexity-theoretic collapse (and making it quite likely that the conjecture indeed holds).
 
\begin{restatable}{theorem}{quasip}\label{thm:quasip}
  There exists a function $f$ such that for every positive integer $k$, \mis can be solved in quasi\-polynomial time $n^{O(k^2 \log n+f(k))}$ in $n$-vertex $\mathcal{O}_k$-free graphs.
 \end{restatable}
 
 This is in sharp contrast with what is deemed possible in general graphs.
 Indeed, any exact algorithm for \smis requires time $2^{\Omega(n)}$ unless the Exponential Time Hypothesis (asserting that solving $n$-variable \textsc{$3$-SAT} requires time $2^{\Omega(n)}$) fails~\cite{Impagliazzo01}.
 
 \medskip
 
 It should be noted that~\cref{conj:mis} is a special case of an intriguing and very general question by Dallard, Milani\v c, and \v Storgel~\cite{Dallard21} of whether there are planar graphs $H$ for which \smis is NP-complete on graphs excluding $H$ as an induced minor.
 In the same paper, the authors show that \smis is in fact polytime solvable when $H$ is $W_4$ (the 4-vertex cycle with a fifth universal vertex), or $K_5^-$ (the 5-vertex clique minus an edge), or $K_{2,t}$. 
Gartland et al.~\cite{Gartland21} (at least partially) answered that question when $H$ is a path, or even a cycle, by presenting in that case a quasi-polynomial algorithm for \smis.
 As we will mention again later, Korhonen~\cite{Korhonen22} showed that bounded-degree graphs excluding a~fixed planar graph~$H$ as an induced minor have bounded treewidth, thereby fully settling the question of Dallard et al. when the degree is bounded. 
 He also derived an algorithm running in time $2^{O(n/\log^{1/6}n)}=2^{o(n)}$, in the general (non bounded-degree) case.
 
 \Cref{thm:quasip} now adds a quasi-polynomial time algorithm when $H$ is the disjoint union of triangles.
 This is an orthogonal generalization of the trivial case when $H$ is a triangle (hence the graphs are forests) to that of Gartland et al.
 We increase the number of triangles, while the latter authors increase the length of the cycle.
 Our proofs are very different, yet they share a common feature, that of measuring the progress of the usual branching on a vertex by the remaining amount of relevant (semi-)induced subgraphs. 
 
 %, for every $\varepsilon > 0$, approximating \smis within ratio $n^{1-\varepsilon}$ is NP-hard~\cite{Hastad96,Zuckerman07}, and for any sufficiently large $r=r(n)$, an $r$-approximation in time $2^{O(n^{1-\varepsilon}/r^{1+\varepsilon})}$ would refute the randomized version of~ETH~\cite{Chalermsook13}. 
%For instance, investing time $2^{O(\sqrt n)}$, one cannot hope for significantly better than a $\sqrt n$-approximation.
 
 \medskip
 
A natural related problem is the complexity of deciding $\mathcal{O}_k$-freeness. 
%
%Surprisingly this problem is open already when $k=2$.
%
%\begin{conjecture}\label{conj:decision}
%Deciding if a graph has two independent cycles is in P.
%\end{conjecture}
A simple consequence of~\cref{cor:main-treewidth} is that one can test whether a graph without $K_{t,t}$ subgraph is $\mathcal{O}_k$-free in polynomial time. For $k=2$, when no complete bipartite is excluded as a subgraph, Le~\cite{Khang17} conjectured the following, which had been raised as an open question earlier by Raymond~\cite{Ray15}.

\begin{conjecture}[Le~\cite{Khang17}]\label{conj:khang}
There is a constant $c$ such that every $\mathcal{O}_2$-free $n$-vertex graph has at most $n^c$ distinct induced paths.
\end{conjecture}

\cref{conj:khang} was recently solved by Nguyen, Scott, and Seymour~\cite{NSS22} in the more general $\mathcal{O}_k$-free case using our \Cref{thm:mainthm}. This implies in particular that the number of induced cycles in $\mathcal{O}_k$-free graphs is polynomial (since this number cannot be more than $n$ times  the number of induced paths), and thus testing $\mathcal{O}_k$-freeness can be done in polynomial time by enumerating all induced cycles and testing, for every $k$ cycles in this collection, whether they are pairwise independent. 

%
%resolves~\cref{conj:decision}.
%\Cref{thm:mainthm} yields a quasi-polynomial upper bound of $n^{O(\log n)}$ on the number of induced paths in $\mathcal{O}_k$-free $t$-biclique-free graphs.
%Indeed, in an $n$-vertex graph~$G$ with feedback vertex set~$X$ of size~$O(\log n)$,
%any induced path~$P$ is entirely described by its vertex set~$V(P)$.
%The latter can be decomposed into $X' = X \cap V(P)$ with size $O(\log n)$,
%and $O(\log n)$ paths in the forest $G-X$, each of which is determined by its endpoints.
%Hence to describe~$V(P)$ it suffice to give~$X'$ and the endpoints of these paths, which is $O(\log n)$ vertices,
%leaving only $n^{O(\log n)}$ choices for an induced paths.

\medskip

\paragraph{Organization of the paper.} In~\cref{sec:lb}, we prove that~\cref{thm:mainthm,cor:main-treewidth} are tight already for $k=2$ and $t=3$. 
\Cref{sec:algo} solves \smis in $\mathcal{O}_k$-free graphs in quasi-polynomial time, among other algorithmic applications of \cref{cor:main-treewidth}.

The proof of our main structural result, \cref{thm:mainthm}, spans from \cref{sec:prel} to~\cref{sec:main}. After some preliminary results (\cref{sec:prel}), we show in \Cref{sec:statement} that it suffices to prove \Cref{thm:mainthm} when the graph $G$ has a simple structure: a cycle $C$, its neighborhood $N$ (an independent set), and the remaining vertices $R$ (inducing a forest). Instead of directly exhibiting a logarithmic-size feedback vertex set, we rather prove that every such graph contains a vertex of degree linear in the so-called ``cycle rank'' (or first Betti number) of the graph. For sparse $\mathcal{O}_k$-free graphs, the cycle rank is at most linear in the number of vertices and decreases by a constant fraction when deleting a vertex of linear degree. We then derive the desired theorem by induction, using as a base case that if the cycle rank is small, we only need to remove a small number of vertices to obtain a tree. %Marthe: lying a little bit here, only true for the reduced graph, is that a big deal?
To obtain the existence of a linear-degree vertex in this simplified setting, we argue in~\Cref{sec:cutting} that we may focus on the case where the forest $G[R]$ contains only paths or only large ``well-behaving''
subdivided stars. In~\cref{sec:tsp_new}, we discuss how the $\mathcal{O}_k$-freeness restricts the adjacencies between these stars/paths and $N$. Finally, in~\cref{sec:main}, we argue that the restrictions yield a simple enough picture, and derive our main result.

\section{Sparse $\mathcal{O}_2$-free graphs with unbounded treewidth}\label{sec:lb}

In this section, we show the following.
\begin{theorem}\label{thm:construction}
For every natural $k$, there is an $\mathcal{O}_2$-free  graph with $2^k+k-1$ vertices, which does not contain $K_{3,3}$ as a subgraph and has treewidth  $k$. 
\end{theorem}
In particular, for infinitely many values of $n$, there is an $\mathcal{O}_2$-free  $n$-vertex graph which does not contain $K_{3,3}$ as a subgraph and has treewidth at least $\log_2 n - 1$.

\paragraph*{Construction of $G_k$.}
To build $G_k$, we first define a word $w_k$ of length $2^k-1$ on the alphabet~$[k]$.
We set $w_1 = 1$, and for every integer $i > 1$, $w_i = i~w_{i-1}[1]~i~w_{i-1}[2]~i~\ldots~i~w_{i-1}[2^{i-1}-2]$ $~i~w_{i-1}[2^{i-1}-1]~i$.
It is worth noting that equivalently $w_i = \text{incr}(w_{i-1})~1~\text{incr}(w_{i-1})$, where $\text{incr}$ adds 1 to every letter of the word. 
Let $\Pi_k$ be the $(2^k-1)$-path where the $\ell$-th vertex of the path (say, from left to right) is denoted by $\Pi_k[\ell]$. 

The graph $G_k$ is obtained by adding to $\Pi_k$ an independent set of $k$ vertices $v_1, v_2, \ldots, v_k$, and linking by an edge every pair $v_i, \Pi_k[\ell]$ such that $i \in [k]$ and $w_k[\ell]=i$.

\smallskip

Observe that we can also define the graph $G_k$ directly, rather than iteratively: it is the union of a path $u_1,\ldots,u_{2^k-1}$  and an independent set $\{v_0,\ldots,v_{k-1}\}$, with an edge between $v_i$ and $u_j$ if and only if $i$ is the 2-order of $j$ (the maximum $k$ such that $2^k$ divides $j$).

See~\cref{fig:lowerbound} for an illustration.  

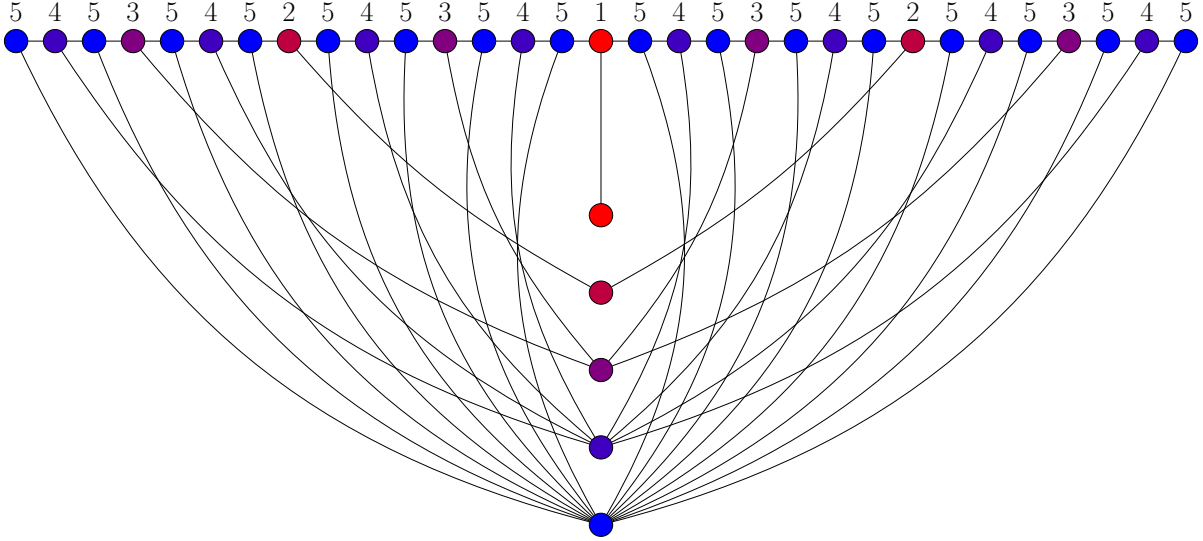
\begin{figure}[h!]
\centering
\resizebox{450pt}{!}{
\begin{tikzpicture}
\def\k{5}
\pgfmathtruncatemacro\km{\k-1}
\def\mw{0.6}
\node[draw,circle,fill=red,minimum width=\mw cm] (u1) at (0,0) {} ;
\node at (0,0.75) {\LARGE{$1$}} ;
\foreach \i in {2,...,\k}{
\pgfmathtruncatemacro\p{100 * \i / \km - 100 / \km}
\pgfmathtruncatemacro\q{2^\i / 2}
\pgfmathtruncatemacro\qm{\q / 2 - 1/2}
\pgfmathtruncatemacro\kk{\k - \i + 1}
\pgfmathtruncatemacro\s{2^\kk}
\pgfmathtruncatemacro\l{\k - \i}
\pgfmathtruncatemacro\ss{2^\l - 1}
\foreach \j in {0,...,\qm}{
\node[draw,circle,fill=blue!\p!red,minimum width=\mw cm] (u\i-\j) at (- \j * \s - \ss - 1,0) {} ;
\node at (- \j * \s - \ss - 1,0.75) {\LARGE{$\i$}} ;
\node[draw,circle,fill=blue!\p!red,minimum width=\mw cm] (v\i-\j) at (\j * \s + \ss + 1,0) {} ;
\node at (\j * \s + \ss + 1,0.75) {\LARGE{$\i$}} ;
}
}
\foreach \i in {1,...,\k}{
\pgfmathtruncatemacro\p{100 * \i / \km - 100 / \km}
\node[draw,circle,fill=blue!\p!red,minimum width=\mw cm] (w\i) at (0,-2.5 - 2 * \i) {} ;
}

\pgfmathtruncatemacro\q{2^\k / 2 - 1}
\foreach \j in {-\q,...,\q}{
  \node[circle,minimum width=\mw cm] (a\j) at (\j,0) {} ;
}

%edges
\pgfmathtruncatemacro\qp{-2^\k / 2 + 2}
\pgfmathtruncatemacro\qm{\q - 1}
\foreach \j [count = \jp from \qp] in {-\q,...,\qm}{
  \draw (a\j) -- (a\jp) ;
}

\draw (w1) -- (u1) ;
\foreach \i in {2,...,\k}{
\pgfmathtruncatemacro\q{2^\i / 2}
\pgfmathtruncatemacro\qm{\q / 2 - 1/2}
\foreach \j in {0,...,\qm}{
\draw (w\i) to [bend left = 5 * \i] (u\i-\j) ;
\draw (w\i) to [bend right = 5 * \i] (v\i-\j) ;
}
}
\end{tikzpicture}
}
\caption{The graph $G_k$ for $k=5$: an $\mathcal{O}_2$-free graph without $K_{3,3}$ subgraph, $k+2^k-1$ vertices, and treewidth~$k$.}
\label{fig:lowerbound}
\end{figure}

\paragraph*{$G_k$ is $\mathcal{O}_2$-free and has no $K_{3,3}$ subgraph.}
The absence of $K_{3,3}$ (even $K_{2,3}$) as a subgraph is easy to check.
At least one vertex of the $K_{3,3}$ has to be some $v_i$, for $i \in [k]$.
It forces that its three neighbors $x, y, z$ are in $\Pi_k$.
In turn, this implies that a common neighbor of $x, y, z$ (other than $v_i$) is some $v_{i'} \neq v_i$; a contradiction since distinct vertices of the independent set have disjoint neighborhoods.

We now show that $G_k$ is $\mathcal{O}_2$-free.
Assume towards a contradiction that $G_k[C_1 \cup C_2]$ is isomorphic to the disjoint union of two cycles $G_k[C_1]$ and $G_k[C_2]$.
As $C_1$ and $C_2$ each induce a cycle, they each have to intersect $\{v_1, \ldots, v_k\}$.
Assume without loss of generality that $C_1$ contains $v_i$, and $C_2$ is disjoint from $\{v_i, v_{i+1}, \ldots, v_k\}$.
Consider a subpath $S$ of $C_2$ with both endpoints in $\{v_1, \ldots, v_k\}$, thus in $\{v_1, \ldots, v_{i-1}\}$, and all the other vertices of $S$ form a set $S' \subseteq V(\Pi_k)$.
It can be that the endpoints are in fact the same vertex $v_{i'}$, and in that case $S$ is the entire $C_2$.

Let $v_{i'}, v_{i''}$ be the two (possibly equal) endpoints.
Observe that $S'$ is a subpath of $\Pi_k$ whose two endpoints have label $i', i'' < i$.
In particular there is a vertex labeled $i$ somewhere along~$S'$.
This makes an edge between $v_i \in C_1$ and $C_2$, which is a contradiction.

\paragraph*{$G_k$ has treewidth $k$.}
Since $\{v_2, \ldots v_k\}$ is a feedback vertex set,  the treewidth of $G_k$ is at most $k$, so it is enough to prove that $G_k$ has treewidth at least $k$. We do this by proving that $G_k$ contains the complete graph $K_{k+1}$ as a minor (we thank an anonymous reviewer for suggesting the argument below, our initial argument only gave a $K_{k}$-minor in $G_k$). The minor $K_{k+1}$ is constructed as follows: for each $i\in [k]$, we denote by $V_i$ the subpath of $\Pi_k$ whose right endpoint is the leftmost vertex of $\Pi_k$ labeled $i$, and which is maximal with the property that it does not contain any vertex labeled $i+1$. Note that each set $V_i$ contains vertices labeled $i, i+2, i+3, \ldots, k$, and is adjacent to a vertex labeled $i+1$. For each $i\in [k]$, we let $V_i'$ be the union of $V_i$ and the vertex $v_i$ (this set induces a connected subgraph of $G_k$), and we define $V_{k+1}'$ as the set of vertices of $\Pi_k$ lying to the right of the unique vertex of $\Pi_k$ labeled 1. Note that the sets $V_i'$, $i\in [k+1]$, form a partition of $V(k+1)$. By definition there is an edge between any two sets $V_i', V_j'$ in $G_k$ for $1\le i<j\le k+1$, and thus $G_k$ contains $K_{k+1}$ as a minor, as desired.

%We build a $K_k$-minor as follows.
%We set $V_1 = \{v_1\} \cup R_1$ where $R_1 \subseteq V(\Pi_k)$ consists of the unique neighbor $x_1$ of $v_1$ together with all the vertices of $\Pi_k$ to the right of $x_1$.
%Then, for every $i \in [2,k]$, we iteratively define $V_i$ as $\{v_i\} \cup R_i$ where $R_i$ is made of $x_i$ the unique neighbor of $v_i$ in $V(\Pi_k) \setminus \bigcup_{1 \leqslant j \leqslant i-1} V_j$ and all the vertices of $V(\Pi_k) \setminus \bigcup_{1 \leqslant j \leqslant i-1} V_j$ to the right of $x_i$.

%By construction, it is clear that the $V_i$'s are disjoint and that $G_k[V_i]$ is a path for every $i \in [k]$, hence is connected.
%It can also be observed that there is an edge between $R_i$ and $v_j$ with $j>i$.
%Thus the $V_i$'s are the branch sets of a $K_k$-minor.
%Therefore $\tw(G_k) \geqslant \tw(K_k) = k-1$.
%(It is easy to see that the treewidth of $G_k$ is at most $k$, since $\{v_2, \ldots v_k\}$ is a feedback vertex set.)

\medskip

The \emph{twin-width} of $G_k$, however, can be shown to be at most a constant independent of~$k$.

\section{Algorithmic applications}\label{sec:algo}
This section presents algorithms on $\mathcal{O}_k$-free graphs based on our main result, specifically using the treewidth bound.
\maintw*

Single-exponential parameterized $O(1)$-approximation algorithms exist for treewidth.
Already in 1995, Robertson and Seymour~\cite{Robertson95} present a~$2^{O(\tw)}n^2$-time algorithm yielding a~tree-decomposition of width $4(\tw+1)$ for any input $n$-vertex graph of treewidth $\tw$.
Run on $n$-vertex graphs of logarithmic treewidth, this algorithm outputs tree-decompositions of width $O(\log n)$ in polynomial time. 
We thus obtain the following.

\begin{corollary}\label{cor:gen-alg}
\textsc{Maximum Independent Set}, \textsc{Hamiltonian Cycle}, \textsc{Minimum Vertex Cover}, \textsc{Minimum Dominating Set}, \textsc{Minimum Feedback Vertex Set}, and \textsc{Minimum Coloring} can be solved in polynomial time $n^{g(t,k)}$  $\mathcal{O}_k$-free graphs with no $K_{t,t}$ subgraph, for some function $g$.
\end{corollary}

\begin{proof}
Let $h(t,k)$ be the implicit function in~\cref{cor:main-treewidth} such that every $\mathcal{O}_k$-free $n$-vertex graph with no $K_{t,t}$ subgraph has treewidth at most $h(t,k) \log n$.

Algorithms running in time $2^{O(\tw)}n^{O(1)}=2^{h(t,k) \log n}n^{O(1)}=n^{h(t,k)+O(1)}=n^{g(t,k)}$ exist for all these problems but for \textsc{Minimum Coloring}.
They are based on dynamic programming over a tree-decomposition, which by \cref{cor:main-treewidth} has logarithmic width and by~\cite{Robertson95} can be computed in polynomial time.
For \textsc{Maximum Independent Set}, \textsc{Minimum Vertex Cover}, \textsc{Minimum Dominating Set}, and \mbox{\textsc{$q$-Coloring}} (for a fixed integer $q$) see for instance the textbook~\cite[Chapter 7.3]{Cygan15}.
For \textsc{Hamiltonian Cycle} and \textsc{Minimum Feedback Vertex Set}, deterministic parameterized single-exponential algorithms require the so-called rank-based approach; see~\cite[Chapter 11.2]{Cygan15}. 

By~\cref{cor:avg-degree}, $\mathcal{O}_k$-free graphs with no $K_{t,t}$ subgraph have bounded chromatic number. 
Thus a~polynomial time algorithm for \textsc{Minimum Coloring} is implied by the one for \textsc{$q$-Coloring}.
\end{proof}

%It would be nice to see if some dedicated algorithm exists for MIS or for hamiltonian cycle.

In a scaled-down refinement of Courcelle's theorem~\cite{Courcelle90}, Pilipczuk showed that any problem expressible in Existential Counting Modal Logic (ECML) admits a single-exponential fixed-parameter algorithm in treewidth~\cite{Pilipczuk11}.
In particular:
\begin{theorem}[\cite{Pilipczuk11}]\label{thm:ecml-logtw}
ECML model checking can be solved in polynomial time on any class with logarithmic treewidth. 
\end{theorem}

In a nutshell, this logic allows existential quantifications over vertex and edge sets followed by a counting modal formula that should be satisfied from every vertex $v$.
Counting modal formulas enrich quantifier-free Boolean formulas with $\Diamond^S \varphi$, whose semantics is that the current vertex~$v$ has a number of neighbors satisfying $\varphi$ in the ultimately periodic set $S$ of non-negative integers. 
Another consequence of~\cref{cor:main-treewidth} (and \cref{thm:ecml-logtw}) is that testing if a graph is $\mathcal{O}_k$-free can be done in polynomial time among sparse graphs, further indicating that the general case could be tractable. cb

\begin{corollary}\label{cor:recognition}
For any fixed $k$ and $t$, deciding whether a  graph with no $K_{t,t}$ subgraph is $\mathcal{O}_k$-free can be done in polynomial time.
\end{corollary}
\begin{proof}
One can observe that $\mathcal{O}_k$-freeness is definable in ECML.
Indeed, one can write $$\varphi = \exists X_1 \exists X_2 \ldots \exists X_k~\left(\bigwedge_{1 \leqslant i \leqslant k} X_i \rightarrow \Diamond^{\{2\}} X_i\right)~\land~\left(\bigwedge_{1 \leqslant i < j \leqslant k} \neg (X_i \land X_j) \land (X_i \rightarrow \Diamond^{\{0\}} X_j)\right).$$
Formula $\varphi$ asserts that there are $k$ sets of vertices $X_1, X_2, \ldots, X_k$ such that every vertex has exactly two neighbors in $X_i$ if it is itself in $X_i$, the sets are pairwise disjoint, and every vertex has no neighbor in $X_j$ if it is in some distinct $X_i$ (with $i < j$).
Thus $G$ is $\mathcal{O}_k$-free if and only if $\varphi$ does not hold in $G$.
\end{proof}

We now show the main algorithmic consequence of our structural result.
This holds for any (possibly dense) $\mathcal{O}_k$-free graph, and uses the sparse case (\cref{cor:gen-alg}) at the basis of an induction on the size of a largest collection of independent 4-vertex cycles. It should be noted that this result (as well as the previous result on \smis above) also works for the weighted version of the problem, with minor modifications. 

\quasip*
\begin{proof}
Let $G$ be our $n$-vertex $\mathcal{O}_k$-free input.
Let $q$ be the maximum integer such that $G$ admits $q$~independent 4-vertex cycles (the cycles themselves need not be induced).
Clearly $q < k$.
We show the theorem by induction on $q$, namely that \smis can be Turing-reduced in time $n^{c (q+1)^2 \log n}$ for some constant $c$ (specified later) to smaller instances with no $K_{2,2}$ subgraphs (hence such that $q=0$).
We first examine what happens with the latter instances.
Let $f(k)=h(2,k)$ with $h(t,k)$ the hidden dependence of~\cref{cor:main-treewidth}. 
If $q=0$, $G$ does not contains $K_{2,2}$ as a subgraph, so we can solve \smis in polynomial time $n^{f(k)+O(1)}$ by~\cref{cor:gen-alg}. 

We now assume that $q \geqslant 1$, $n \geqslant 4$, and that the case $q-1$ of the induction has been established (or $q-1=0$).
Let $C$ be a 4-vertex cycle part of a $4q$-vertex subset consisting of $q$~independent \mbox{4-vertex} cycles.
%Imagine a packing, extending $A_1$, of \mbox{4-vertex} cycles in $G$, say, $A_1, \ldots, A_h$ until $G - \bigcup_{i \in [h]} A_i$ has no 4-vertex cycle.
Let $\mathcal S$ be the set of all $4q$-vertex subsets consisting of $q$~independent 4-vertex cycles in the current graph (at this point, $G$), and $s = |\mathcal S|$.
Thus $1 \leqslant s \leqslant n^{4q}$.
By assumption, the closed neighborhood of $C$, $N[C]$, intersects every subset in $\mathcal S$.
In particular, there is one of the four vertices of $C$, say, $v$, such that $N[v]$ intersects at least $s/4$ subsets of $\mathcal S$. 

We branch on two options: either we put $v$ in (an initially empty set) $I$, and remove its closed neighborhood from $G$, or we remove $v$ from $G$ (without adding it to $I$).
With the former choice, the size of $\mathcal S$ drops by at least $s/4$, whereas with the latter, it drops by at least 1.

Even if fully expanded while $s >0$, this binary branching tree has at most 
$$\sum\limits_{0 \leqslant i \leqslant 4q \log_{4/3}n} {n \choose i} = n^{O(q \log n)}~\text{leaves,}$$
since including a vertex in $I$ can be done at most~$4q \log_{4/3}n$ times within the same branch; thus, leaves can be uniquely described as binary words of length~$n$ with at most~$4q \log_{4/3}n$ occurrences of, say, 1.

We retrospectively set $c \geqslant 1$ such that the number of leaves is at most $n^{c q \log n}$, running the algorithm thus far (when $q \geqslant 1$) takes at most time $n^{c+c q \log n}$.
At each leaf of the branching, $s = 0$ holds, which means that the current graph does not admit $q$~independent 4-vertex cycles.
By the induction hypothesis, we can Turing-reduce each such instance in time $n^{c q^2 \log n}$.
Thus the overall running time is $$n^{c+c q \log n} + n^{c q \log n} \cdot n^{c q^2 \log n} \leqslant n^{c+c q \log n} \cdot (n^{c q^2 \log n}+1) \leqslant n^{c(q+1)^2 \log n - cq \log n - c \log n + c + \frac{1}{\log n}}.$$
Note that $n^{cq^2 \log n} \geqslant 1$ thus we could upper-bound $n^{cq^2 \log n}+1$ by $2n^{cq^2 \log n}=n^{cq^2 \log n + \frac{1}{\log n}}$.
Since $c, q \geqslant 1$ and $\log n \geqslant 1$, it holds that $-cq \log n - c \log n + c + \frac{1}{\log n} \leqslant -2c+c+1 \leqslant 0$.
Hence we get the claimed running time of $n^{c(q+1)^2 \log n}$ for the reduction to $q=0$, and the overall running time of $n^{c(q+1)^2 \log n + f(k) + O(1)}=n^{O(k^2 \log n+f(k))}$.
\end{proof}

One may wonder if some other problems beside \smis become (much) easier on $\mathcal{O}_k$-free graphs than in general.
As $2K_2$-free graphs are $\mathcal{O}_2$-free, one cannot expect a quasi-polynomial time algorithm for \textsc{Minimum Dominating Set}~\cite{Bertossi84,Corneil84}, \textsc{Hamiltonian Cycle}~\cite{Golumbic04}, \textsc{Maximum Clique}~\cite{Poljak74}, and \textsc{Minimum Coloring}~\cite{Kral01} since these problems remain NP-complete on $2K_2$-free graphs.
Nevertheless we give a quasi-polynomial time algorithm for \textsc{3-Coloring}.

\begin{theorem}
  There exists a function $f$ such that for every positive integer $k$, \textsc{3-Coloring} can be solved in quasi-polynomial time $n^{O(k^2 \log n+f(k))}$ in $n$-vertex $\mathcal{O}_k$-free graphs.
 \end{theorem}
 \begin{proof}
 We solve the more general \textsc{List 3-Coloring} problem, where, in addition, every vertex~$v$ is given a \emph{list} $L(v) \subseteq \{1,2,3\}$ from which one has to choose its color.
 Note that when $L(v)=\emptyset$ for some vertex $v$, one can report that the instance is negative, and when $|L(v)|=1$, $v$ has to be colored with the unique color in its list, and this color has to be deleted from the lists of its neighbors (once this is done, $v$ might as well be removed from the graph).
These reduction rules are performed as long as they apply, so we always assume that the current instance has only lists of size 2 and 3.

We follow the previous proof, and simply adapt the branching rule, and the value of $s$.
Now $s$ is defined as the sum taken over all vertex sets $X$ consisting of $q$ independent 4-vertex cycles (the cycles themselves need not be induced), of the sum of the list sizes of the vertices of $X$.
Hence $8 \leqslant s \leqslant 12 \cdot n^{4q}$.
There is a vertex $v \in C$ and a~color $c \in L(v)$ such that $c$ appears in at least $\frac{1}{2} \cdot \frac{1}{12} \cdot \frac{s}{4}=\frac{s}{96}$ of the lists of its neighbors.
This is because all the lists have size at least 2, and are subsets of $\{1,2,3\}$, thus pairwise intersect.
(Note that this simple yet crucial fact already breaks down for \textsc{List 4-Coloring}.)

We branch on two options: either we color $v$ with $c$, hence we remove color $c$ from the lists of its neighbors or we commit to not color $v$ by $c$, and simply remove $c$ from the list of $v$.
With the former choice, the size of $\mathcal S$ drops by at least $s/96$, whereas with the latter, it drops by at least 1.
The rest of the proof is similar with a possibly larger constant $c$.
 \end{proof}

\section{Preliminary results}\label{sec:prel}

An important property of graphs which do not contain the complete bipartite graph $K_{t,t}$ as a subgraph is that they are not dense (in the sense that they have a subquadratic number of edges). 

\begin{theorem}[K\H{o}v{\'a}ri, S{\'o}s, and Tur{\'a}n~\cite{KST}]\label{thm:kst}
For every integer $t\ge 2$ there is a constant $c_t$ such that any $n$-vertex graph with no $K_{t,t}$ subgraph has at most $c_t \,n^{2-1/t}$ edges.
\end{theorem}

%The following is sometimes known as the Caro-Wei inequality.

%\begin{theorem}[Caro~\cite{Car79} and Wei~\cite{Wei81}]\label{thm:CW}
%Every $n$-vertex graph of average degree at most $d$ contains an independent set of size at least $\tfrac{n}{d+1}$.
%\end{theorem}

The following lemma shows that for $\mathcal{O}_k$-free graphs, excluding $K_{t,t}$ as a subgraph is equivalent to a much stronger `large girth' condition, up to the removal of a bounded number of vertices.

\begin{lemma}
  \label{lem:sparsification}
  There is a function $f$ such that for any integer $\ell$ and any $\mathcal{O}_k$-free graph $G$ with no $K_{t,t}$ subgraph, the maximum number of vertex-disjoint cycles of length at most $\ell$ in $G$ is at most $f(\ell,t,k)$.
\end{lemma}
\begin{proof}
If $\ell\le 2$, we define $f(\ell,t,k)=0$ for any integers $t$ and $k$, and we observe that since $G$ does not contain any cycle of length at most $\ell$, the statement of the lemma holds trivially. 

Assume now that $\ell\ge 3$, and define $f(\ell,t,k):=(2c_t k\ell^2)^t$, where $c_t$ is the constant of Theorem \ref{thm:kst}.

Assume for the sake of contradiction that $G$ contains $N:=f(\ell,t,k)$ vertex-disjoint cycles of length at most $\ell$, which we denote by $C_1,\ldots,C_N$. Let $H$ be the graph with vertex set $v_1,\ldots,v_N$, with an edge between $v_i$ and $v_j$ in $H$ if and only if there is an edge between $C_i$ and $C_j$ in $G$. Since $G$ is $\mathcal{O}_k$-free, $H$ has no independent set of size $k$. By Tur\'an's theorem~\cite{Tur41}, $H$ contains at least $\tfrac{N^2}{2k-2}-\tfrac{N}2\ge \tfrac{N^2}{2k}-\tfrac{N}2$ edges.

Consider the subgraph $G'$ of $G$ induced by the vertex set  $\bigcup_{i=1}^N C_i$. The graph $G'$ has $n\le \ell N$ vertices, and $m\ge 3N +\tfrac{N^2}{2k}-\tfrac{N}2> \tfrac{N^2}{2k}$ edges. Note that by the definition of $N$, we have 
\[
m>\tfrac{N^2}{2k}=\tfrac1{2k}\cdot N^{2-1/t}\cdot N^{1/t}\ge
\tfrac1{2k\ell^{2-1/t}}\cdot n^{2-1/t}\cdot 2c_t k\ell^2\ge c_t \,n^{2-1/t},
\]
which contradicts Theorem~\ref{thm:kst}, since $G'$ (as an induced subgraph of $G$) does not contain $K_{t,t}$ as a subgraph.
\end{proof}

The \emph{girth}  of a graph $G$ is the minimum length of a cycle in $G$ (if $G$ is acyclic, its girth is set to be infinite). We obtain the following immediate corollary of Lemma~\ref{lem:sparsification}. 

\begin{corollary} \label{cor:sparsification}
There is a function $g$ such that for any integer $\ell\ge 3$, any $\mathcal{O}_k$-free graph $G$ with no $K_{t,t}$ subgraph contains a set $X$ of at most $g(\ell,t,k)$ vertices such that $G-X$ has girth at least~$\ell$.
\end{corollary}

\begin{proof}
Let $f$ be the function of Lemma~\ref{lem:sparsification}, and let $g(\ell,t,k):= (\ell-1)\cdot f(\ell-1,t,k)$. Consider a maximum collection of disjoint cycles of length at most $\ell-1$ in $G$. Let $X$ be the union of the vertex sets of all these cycles. By Lemma~\ref{lem:sparsification}, $|X|\le (\ell-1) f(\ell-1,t,k)=g(\ell,t,k)$, and by definition of $X$, the graph $G-X$ does not contain any cycle of length at most $\ell-1$, as desired.
\end{proof}

We now state a simple consequence of Corollary \ref{cor:sparsification}, which will be particularly useful at the end of the proof of our main result. A \emph{banana} in a graph $G$ is a pair of vertices joined by at least 2 disjoint paths whose internal vertices all have degree 2 in $G$.

\begin{corollary}\label{corollary:banana}
There is a function $f'$ such that  any $\mathcal{O}_k$-free graph $G$ with no $K_{t,t}$ subgraph contains a set $X$ of at most $f'(t,k)=O_t(k^t)$ vertices such that all bananas of $G$ intersect $X$.
\end{corollary}

\begin{proof}
Let $G'$ be the graph obtained from $G$ by replacing each maximal path whose internal vertices have degree 2 in $G$ by a path on two edges (with a single internal vertex, of degree 2). Note that each banana in $G$ is replaced by a copy of some graph $K_{2,s}$ in $G'$, with $s\ge 2$. In particular, every set $X'\in V(G')$ intersecting all 4-cycles in $G'$ intersects all copies of graphs $K_{2,s}$ with $s\ge 2$. Moreover, any such set $X'$ in $G'$ can be lifted to a set $X\in V(G)$ of the same size that intersects all bananas of $G$. The result then follows from the application of Corollary \ref{cor:sparsification} to $G'$ with $\ell=5$.
\end{proof}

%The proof is the same as that of Lemma~\ref{lem:sparsification}. We can take $f'(t,k)=2 f(2,t,k)=O(k^t)$, where $f$ is the function of Lemma~\ref{lem:sparsification} and the implicit multiplicative constant in the $O(\cdot)$ depends on $t$. 

In all the applications of Corollary~\ref{corollary:banana}, $t$ will be a small constant (2 or 3).

\medskip

The \emph{average degree} of a graph $G=(V,E)$, denoted by $\mathrm{ad}(G)$, is defined as $2|E|/|V|$.  Let us now prove that $\mathcal{O}_k$-free graphs with no $K_{t,t}$ subgraph have bounded average degree. This can also be deduced from the main result of~\cite{KO04}, but we include a short proof for the sake of completeness. Moreover, the decomposition used in the proof  will be used again in the proof of our main result.

\begin{lemma}
  \label{lem:avg-degreeg11}
 Every $\mathcal{O}_k$-free graph $G$ of girth at least 11 has average degree at most $2k$.
\end{lemma}

\begin{proof}
We proceed by induction on $k$. When $k = 1$, $G$ is a forest, with average degree less than $2$.
  Otherwise, let $C$ be a cycle of minimal length in $G$.
  Let $N$ be the neighborhood of $C$, let $S$ the second neighborhood of $C$, and let $R = V(G) \setminus (C \cup N)$.
  Thus $V(G)$ is partitioned into $C,N,R$, and we have $S \subseteq R$.
  Observe that there are no edges between $C$ and $R$ in $G$, so it follows that $G[R]$ is $\mathcal{O}_{k-1}$-free, and thus $\mathrm{ad}(G[R])\le 2k-2$ by induction. Observe also that since $G$ has girth at least $11$ and $C$ is a minimum cycle, the two sets $N$ and $S$ are both independent sets. Moreover
  each vertex of $N$ has a unique neighbor in $C$, and each vertex in $S$ has a unique neighbor in $N$.
  Indeed, in any other case we obtain a path of length at most 5 between two vertices of $C$, contradicting the minimality of $C$.
  It follows that $C$ is the only cycle in $G[C \cup N \cup S]$, hence this graph has average degree at most $2$.
  As a consequence, $G$ has a partition of its edges into two subgraphs of average degree at most $2k-2$ and at most 2, respectively, and thus $\textrm{ad}(G)\le 2k-2+2=2k$, as desired.
\end{proof}

It can easily be deduced from this result that every $\mathcal{O}_k$-free graph with no $K_{t,t}$ subgraph has average degree at most $h(t,k)$, for some function $h$ (and thus chromatic number at most $h(t,k)+1$). 

\begin{corollary}
  \label{cor:avg-degree}
  There is a function $h$ such that every $\mathcal{O}_k$-free graph with no $K_{t,t}$ subgraph has average degree at most $h(t,k)$, and chromatic number at most  $h(t,k)+1$.
\end{corollary}

\begin{proof}
  Let $G$ be an $\mathcal{O}_k$-free graph that does not contain $K_{t,t}$ as a subgraph. By Corollary \ref{cor:sparsification}, $G$ has a set $X$ of at most $g(11,t,k)$ vertices such that $G-X$ has girth at least 11. Note that $\mathrm{ad}(G)\le \mathrm{ad}(G-X)+|X|\le \mathrm{ad}(G-X)+ g(11,t,k)\le 2k+g(11,t,k)$, where the last inequality follows from \cref{lem:avg-degreeg11}. 
  
  Let $h(t,k)=2k+g(11,t,k)$. As the class of $\mathcal{O}_k$-free graphs with no $K_{t,t}$ subgraph is closed under taking induced subgraphs, it follows that any graph in this class is $h(t,k)$-degenerate, and thus $(h(t,k)+1)$-colorable.
\end{proof}

We would like to note that using a result of \cite{Dvo18}, extending earlier results of~\cite{KO04}, it can be proved that the class of $\mathcal{O}_k$-free graphs with no $K_{t,t}$ subgraph actually has \emph{bounded expansion}, which is significantly stronger than having bounded average degree. This will not be needed in our proofs, and it can also be deduced from our main result, as it implies that sparse $\mathcal{O}_k$-free graphs have logarithmic separators, and thus polynomial expansion.

\medskip

A \emph{feedback vertex set} (FVS) $X$ in a graph $G$ is a set of vertices of $G$ such that $G-X$ is acyclic. The minimum size of a feedback vertex set in $G$ is denoted by $\mathrm{fvs}(G)$.
The  classical Erd\H os-P\'osa theorem~\cite{EP65} states that graphs with few vertex-disjoint cycles have small feedback vertex sets.

\begin{theorem}[Erd\H os and P\'osa \cite{EP65}]\label{thm:EP}
There is a constant $c>0$ such that if a multigraph $G$ contains less than $k$ vertex-disjoint cycles, then $\mathrm{fvs}(G)\le c k \log k$.
\end{theorem}

We use this result to deduce the following useful lemma.

\begin{lemma}
  \label{lem:cycle-handles}
There is a constant $c>0$ such that the following holds. Let $G$ consist of a cycle $C$, together with $\ell$ paths $P_1,\ldots,P_\ell$ on at least 2 edges 
\begin{itemize}
    \item whose endpoints are in $C$, and
    \item whose internal vertices are disjoint from $C$, and 
    \item such that the internal vertices of each pair of different paths $P_i,P_j$ are pairwise distinct and non-adjacent.
\end{itemize}
  Suppose moreover that $G$ is $\mathcal{O}_k$-free (with $k\ge 2$) and has maximum degree at most $d+2$.
  Then  
  \[ \ell \le c\, d\, k \log k.\]
\end{lemma}
\begin{proof}
Observe that each path $P_i$ intersects or is adjacent to at most $2(d-1)+4d<6d$ other paths $P_j$:
indeed, if~$P_i$ has endpoints~$x,y$ in~$C$, then there are at most~$2(d-1)$ paths~$P_j$ which intersect~$P_i$ by sharing~$x$ or~$y$ as endpoint,
and at most~$4d$ paths~$P_j$ which are adjacent to~$P_i$ because some endpoint of~$P_j$ is adjacent to either~$x$ or~$y$.
It follows that there exist $s\ge \tfrac{\ell}{6d}$ of these paths, say $P_1,\ldots,P_s$ without loss of generality, that are pairwise non-intersecting and non adjacent.

Consider the subgraph~$G'$ of~$G$ induced by the union of $C$ and the vertex sets of the paths $P_1,\ldots,P_s$.
Since the paths $P_i$, $1\le i \le s$, are pairwise independent,
and since $G'$ does not contain $k$ independent cycles, the graph $G'$ does not contain $k$ vertex-disjoint cycles.
Let $G''$ be the multigraph obtained from $G'$ by suppressing all vertices of degree 2
(i.e., replacing all maximal paths whose internal vertices have degree 2 by single edges).
Observe that since $G'$ does not contain $k$ vertex-disjoint cycles, the graph $G''$ does not contains $k$ vertex-disjoint cycles either.
Observe also that $G''$ is cubic and contains $2s$ vertices.
It was proved by Jaeger~\cite{Jae74} that any cubic multigraph $H$ on $n$ vertices satisfies $\fvs(H)\ge \tfrac{n+2}4$.
As a consequence, it follows from \cref{thm:EP} that  $\frac{2s+2}4 \le \fvs(G'')\le c'k \log k$ (for some constant $c'$), and thus $\ell\le 12 d c' k \log k=c d k\log k$ (for $c=12c'$), as desired.
\end{proof}

A \emph{strict} subdivision of a graph is a subdivision where each edge is subdivided at least once.

\begin{lemma}
  \label{lem:min-degree}
  There is a constant $c>0$ such that for any integer $k\ge 2$, any strict subdivision of a graph of average degree at least $c\, k \log k$ contains a graph of the family $\mathcal{O}_k$ as an induced subgraph.
\end{lemma}

\begin{proof}
Note that if a graph $G$ contains $k$ vertex-disjoint cycles, then any strict subdivision of $G$ contains an induced $\mathcal{O}_k$. Hence, it suffices to prove that any graph with less than $k$ vertex-disjoint cycles has average degree at most  $ck\log k$, for some constant $c$. By \cref{thm:EP}, there is a constant $c'$ such that any graph $G$ with less than $k$ vertex-disjoint cycles contains a set $X$ of at most $c' k\log k$ vertices such that $G-X$ is acyclic. In this case $G-X$ has average degree at most 2, and thus $G$ has average degree at most $c' k \log k+2\le c k\log k$ (for some constant $c$), as desired.
\end{proof}

\section{Logarithmic treewidth of sparse $\mathcal{O}_k$-free graphs}\label{sec:statement}
Recall our main result.
\mainthm*

The proof of \cref{thm:mainthm} relies on \emph{the cycle rank},
which is defined as $r(G) = \card{E(G)} - \card{V(G)} + \card{C(G)}$
where~$C(G)$ denotes the set of connected components of~$G$.
The cycle rank is exactly the number of edges of~$G$ which must be deleted to make~$G$ a forest,
hence it is a trivial upper bound on the size of a minimum feedback vertex set.
Remark the following simple properties.
\begin{lemma}
  The cycle rank is invariant under the following operations:
  \begin{enumerate}
    \item Deleting a vertex of degree~1.
    \item Deleting a connected component which is a tree (and in particular, deleting a vertex of degree~0).
  \end{enumerate}
  \label{lem:cycle-rank}
\end{lemma}
We call \emph{reduction} the operation of iteratively deleting vertices of degree~0 or~1,
which preserves cycle rank by the above lemma.
A graph is \emph{reduced} if it has minimum degree at least~2,
and the \emph{core} of a graph~$G$ is the reduced graph obtained by applying reductions to~$G$ as long as possible.
The inclusion-wise minimal FVS of~$G$ and of its core are exactly the same.

In a graph~$G$, a vertex~$x$ is called $\eps$-\emph{rich} if $d(x) \ge \eps \cdot r(G)$.
Our strategy to prove \cref{thm:mainthm} is to iteratively reduce the graph, find an $\eps$-rich vertex, add it to the FVS and delete it from the graph.
%Our strategy to prove \cref{thm:mainthm} is to iteratively delete $\eps$-rich vertices,
%adding them to the FVS, and then reduce the graph.
The following lemma shows that the cycle rank decreases by a constant factor each iteration,
implying that the process terminates in logarithmically many steps.
\begin{lemma}\label{lem:deletion-cyclerank}
  In a reduced $\mathcal{O}_k$-free graph, deleting a vertex of degree $d$
  decreases the cycle rank by at least $\frac{d-k+1}{2}$.
\end{lemma}
\begin{proof}
  In any graph~$G$, deleting a vertex~$x$ of degree~$d$ decreases the cycle rank by~$d-c$,
  where~$c$ is the number of connected components of~$G-x$ which contain a neighbor of~$x$.
  If~$G$ is $\mathcal{O}_k$-free, then all but at most~$k-1$ components of~$G-x$ are trees.
  Furthermore, if~$T$ is a connected component of~$G-x$ which is a tree,
  then~$T$ must be connected to~$x$ by at least two edges,
  as otherwise~$T$ must contain a vertex of degree~1 in~$G$,
  which should have been deleted during reduction.
  Thus we have
  \begin{equation}
    2c - (k-1) \le d.
  \end{equation}
  Therefore the cycle rank decreases by at least $d - \frac{d+k-1}{2} = \frac{d-k+1}{2}$ as desired.
\end{proof}

The existence of rich vertices is given by the following result.
\begin{theorem}
  \label{thm:Okfree-richvertex}
  For any~$k$, there is some $\eps_k > 0$ such that
  any $\mathcal{O}_k$-free graph with girth at least~$11$ has an $\eps_k$-rich vertex.
\end{theorem}

Let us first prove \cref{thm:mainthm} using \cref{thm:Okfree-richvertex}.
\begin{proof}[Proof of \cref{thm:mainthm}]
  Fix~$k$ and~$t$.
  Given a graph $G$ which is $\mathcal{O}_k$-free and does not contain $K_{t,t}$ as a subgraph,
  we apply \cref{lem:sparsification} to obtain a set~$X$ of size at most~$f(11,t,k)$
  such that $G' \eqdef G - X$ has girth at least~11.
  Thus, it suffices to prove the result for~$G'$, and finally add~$X$ to the resulting FVS of~$G'$.
  Since $\log r(G')\le \log {{|V(G')|}\choose 2}\le 2\log  |V(G')|$,
  we have reduced the problem to the following.
  %\footnotemark
  \begin{claim}
    For any~$k$, there is a constant~$c_k$ such that
    if~$G$ is an $\mathcal{O}_k$-free graph with girth at least~11, then~$\fvs(G) \le c_k \cdot \log r(G)$. %\louis{More precisely we obtain $\fvs(G) \le \tfrac{4}{\eps_k} \cdot \log r(G)+O(k^2+\tfrac{k}{\eps_k})$}
  \end{claim}
 % \footnotetext{In fact $r(G')$ is linear in $\card{V(G')}$ because $G'$ is degenerate, but this is not needed here.}

  Let us now assume that~$G$ is as in the claim, and consider its core~$H$, for which $r(H)=r(G)$ and $\fvs(H)=\fvs(G)$.
  Consider an $\eps_k$-rich vertex~$x$ in~$H$ with~$\eps_k$ as in \cref{thm:Okfree-richvertex}.
  If~$r(G) \ge 2k \cdot \eps_k^{-1}$, then~$d(x) \ge 2k$,
  hence by Lemma~\ref{lem:deletion-cyclerank}, deleting $x$ decreases the cycle rank of $G$ by at least
  \begin{equation}
    \frac{d(x)-k+1}{2} \ge \frac{d(x)}{4} \ge \frac{\eps_k}{4} r(G).
  \end{equation}

  Thus, as long as the cycle rank is more than $2k \cdot \eps_k^{-1}$,
  we can find a vertex whose deletion decreases the cycle rank by a constant multiplicative factor.
  After logarithmically many steps, we have $\fvs(G) \le r(G) \le 2k \cdot \eps_k^{-1}$.
  In the end, the feedback vertex set consists of at most~$f(11,t,k)$ vertices in~$X$,
  logarithmically many rich vertices deleted in the induction,
  and at most $2k \cdot \eps_k^{-1}$ vertices for the final graph.
\end{proof}

We now focus on proving \cref{thm:Okfree-richvertex}.
Let~$G$ be an $\mathcal{O}_k$-free graph with girth at least~11.
Consider~$C$ a shortest cycle of~$G$, $N$ the neighborhood of~$C$,
and $R \eqdef G - (C \cup N)$ the rest of the graph (see Figure~\ref{fig:structure}). 
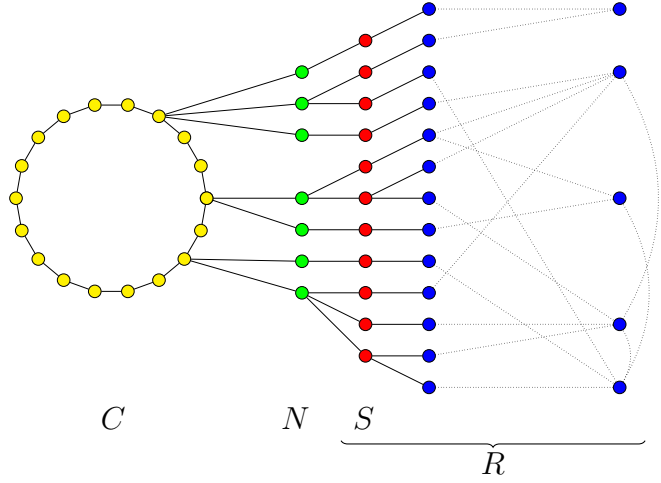
\begin{figure}[t]
\centering
\resizebox{250pt}{!}{
\begin{tikzpicture}
\def\mw{0.05}
\foreach \i in {%
0,20,40,60,80,100,120,140,160,180,200,220,240,260,280,300,320,340,360}{\node[draw,circle,fill=yellow,minimum width=\mw cm] (u\i) at (\i:3cm){} ;}
%\foreach \i in {%
%110,120,130,230,240,250}{
%\draw (\i:3cm) circle (2pt);}
%circle
\foreach \i [evaluate=\i as \j using int(\i+20)] in {%
0,20,40,60,80,100,120,140,160,180,200,220,240,260,280,300,320,340}{
\draw (u\i)--(u\j);}
\foreach \i in {%
4,5,6}
%nodes N
{\node[draw,circle,fill=green,minimum width=\mw cm] (v\i) at (6,\i-2){} ;}
\foreach \i in {%
0,1,2,3}
{\node[draw,circle,fill=green,minimum width=\mw cm] (v\i) at (6,\i-3){} ;}
\foreach \i in {%
0,1,2,3,...,10}
%nodes S
{\node[draw,circle,fill=red,minimum width=\mw cm] (w\i) at (8,\i-5){} ;}
\foreach \i/\j in {%
60/6,60/5,60/4,0/3,0/2,320/1,320/0}{
\draw(u\i) -- (v\j);
}
\foreach \i/\j in {%
0/0,0/1,0/2,1/3,2/4,3/5,3/6,4/7,5/8,5/9,6/10}{
\draw (v\i) -- (w\j);
}
\foreach \i in {%
0,1,2,3,...,12}{
\node[draw,circle,fill=blue,minimum width=\mw cm](z\i) at (10,\i-6){};
}
\foreach \i/\j in {%
0/0,0/1,1/2,2/3,3/4,4/5,5/6,5/7,6/8,7/9,8/10,9/11,10/12}{
\draw(w\i) -- (z\j);
}
\foreach \i in {%
0,2,6,10,12}{
\node[draw,circle,fill=blue,minimum width=\mw cm](p\i) at (16,\i-6){};
}
\foreach \i/\j in {%
0/0,1/2,2/2,3/10,4/0,5/6,6/2,8/6,7/10,8/10,9/10,10/0,11/12,12/12}{
\draw[dotted] (z\i)--(p\j);
}
\draw[dotted] (p0) to [bend right = 30] (p2);
\draw[dotted] (p0) to [bend right = 30] (p6);
\draw[dotted] (p2) to [bend right = 30] (p10);
\end{tikzpicture}
}

\begin{tikzpicture}
\node at (0,0){};
\node at (1,0){$C$};
\node at (3.4,0){$N$};
\node at (4.3,0){$S$};
\draw [decorate,
    decoration = {brace}] (8,-0.3) --  (4,-0.3);
\node at (6,-0.6){$R$};    
\end{tikzpicture}
\caption{Subgraph of an $\mathcal{O}_4$-free graph $G$. 
%\ed{direct paths? but too early}. 
$V(G)$ is partitioned into three sets $C,N,R$, where $C$ is a shortest cycle, $N$ is an independent set and first neighborhood of $C$, and $R$ is $\mathcal{O}_{3}$-free. 
$S$ is the second neighborhood of $N$. Gray lines correspond to induced paths where all internal vertices have degree 2.} %\textcolor{blue}{A:I thought it looked weird as there were only neighbors in $N$ on one side of the graph, not sure how to make it nicely. Any idea, please let me know and I am happy to change.:) Or feel free to change yourself.}}\ed{I see. I think it's fine if the edges incident to $C$ are not evenly distributed for the sake of the legibility/planarity of the figure. In general, the figure is nice and helpful. I'm just wondering where to place it. $S$ and direct paths are defined/used only (much) later.}
%
%\colin{The \texttt{[h!]} gives an awkward break in the middle of the paragraph.
%  Is it okay to let the figure float and add a reference in the paragraph?}
\label{fig:structure}
\end{figure}
Remark that there is no edge between $C$ and $R$, hence $G[R]$ is an $\mathcal{O}_{k-1}$-free graph.
As a special case, if $k=2$, then~$G[R]$ is a forest.
We will show that in general, it remains possible to reduce the problem to the case where~$G[R]$ is a forest,
which is our main technical theorem.
\begin{theorem}
  For any~$k$, there is some $\delta_k > 0$ such that
  if~$G$ is a connected $\mathcal{O}_k$-free graph with girth at least~$11$,
  and furthermore~$G[R]$ is a forest where~$R$ is as in the decomposition described above,
  then~$G$ has a $\delta_k$-rich vertex.
 \label{thm:richvertex}
\end{theorem}

\Cref{thm:richvertex} will be proved in \Cref{sec:main}.
In the remainder of this section, we assume \cref{thm:richvertex}
and explain how \cref{thm:Okfree-richvertex} can be deduced from it.
\begin{proof}[Proof of \cref{thm:Okfree-richvertex}]
  The proof is by induction on~$k$.
  Let~$\delta_k > 0$ be as in \cref{thm:richvertex},
  and let~$\eps_{k-1} > 0$ be as in \cref{thm:Okfree-richvertex}, obtained by induction hypothesis.
  We fix
  \begin{equation}
    \eps_k \eqdef \min\left\{
      \frac{\eps_{k-1}}{20},
      \frac{\delta_k}{20},
      \frac{\delta_k}{5(k+1)},
      \frac{1}{30(k-2)}
    \right\}.
  \end{equation}

  Let~$G$ be any $\mathcal{O}_k$-free graph with girth at least~11.
  Reductions preserve all the hypotheses of the claim, and the value of~$r(G)$,
  hence we can assume~$G$ to be reduced.
  Consider the decomposition~$C,N,R$ as previously described.
  We construct a subset $F \subset R$ inducing a~rooted forest in $G$ such that
  the only edges from~$F$ to $R \setminus F$ are incident to roots of~$F$,
  and each root of~$F$ is incident to at most one such edge.
  \begin{claim}
    If $F \subset R$ has the former property and $F' \subset R \setminus F$ induces a forest in $G$,
    then $F \cup F'$ induces a forest in $G$.
    \label{claim:roots-F}
  \end{claim}
  \begin{proof}
    Each connected component of~$G[F]$ has a single root, which is the only vertex which can be connected to~$F'$.\cqed
  \end{proof}
  We construct~$F$ inductively, starting with an empty forest,
  and applying the following rule as long as it applies:
  if~$x \in R \setminus F$ is adjacent to at most one vertex in $R \setminus F$,
  we add~$x$ to~$F$, and make it the new root of its connected component in~$F$.
  The condition on $F$ obviously holds for $F = \emptyset$.
  When adding~$x$, by \cref{claim:roots-F}, $F \cup \{x\}$ is still a forest.
  Furthermore, if~$y \in F \cup \{x\}$ is adjacent to $R \setminus (F \cup \{x\})$,
  then either $y = x$ or~$y$ was a root before the addition of~$x$, and is not adjacent to~$x$,
  and therefore~$x$ and~$y$ are in distinct connected components of~$F \cup \{x\}$.
  In either case, $y$ is a root of~$F \cup \{x\}$ as required.

  We now denote by~$F$ the forest obtained when the previous rule no longer applies, and let $R' = R \setminus F$. As observed by a reviewer, $R'$ and $F$ can be  defined equivalently by saying that $R'$ is the core of $G[R]$ and $F$ is equal to $R\setminus R'$ (however the procedure described above will be  useful in order to prove the next claims).
  Remark that it might be the case that~$F = R$,
  meaning that~$G[R]$ is a forest (and we fall in the case of \cref{thm:richvertex}),
  or $F = \emptyset$, which means that~$G[R]$ has minimum degree at least~2.
  \begin{claim}
    \label{clm:min-degree-R}
    All vertices in~$G[R']$ have degree at least 2.
  \end{claim}
  \begin{proof}
    A vertex of degree less than~2 in~$G[R']$ should have been added to~$F$.\cqed
  \end{proof}
  \begin{claim}
    \label{clm:CNF-connected}
    The graph~$G[C \cup N \cup F]$ is connected.
  \end{claim}
  \begin{proof}
    It suffices to show that each connected component $T$ of~$G[F]$ is connected to~$N$. Each such component $T$ is a tree. If $T$ consists of a single vertex $v$, then $v$ is the root of $T$ and has at most one neighbor in $R'$ by definition. Since~$G$ is reduced, $v$ has degree at least~2 in~$G$, hence it must be connected to~$N$.

    If $T$ contains at least two vertices, then it contains at least 2 leaves, and in particular at least one leaf $v$ which is not the root of $T$. The vertex $v$ has a single neighbor in $R$ (its parent in $T$), and thus similarly as above it must have a neighbor in $N$.\cqed
  \end{proof}

  Define~$B$ as the set of vertices of~$R'$ adjacent to $N \cup F$,
  and let~$A$ be the set of edges between $N \cup F$ and~$B$.

  \begin{claim}
    If~$\card{A} \le \frac{9}{10} r(G)$, then~$G$ has an $\eps_k$-rich vertex.
  \end{claim}
  \begin{proof}
    Deleting~$A$ from~$G$ decreases the cycle rank by at most~$\card{A}$, hence $r(G-A) \ge r(G)/10$.
    Since~$G[C \cup N \cup F]$ and~$G[R']$ are unions of connected components of~$G-A$, we have
    \[ r(G-A) = r(G[C \cup N \cup F]) + r(G[R']). \]
    Thus either~$G[C \cup N \cup F]$ or~$G[R']$ has cycle rank at least~$r(G) / 20$.
    %\ed{strictly speaking $(r(G)/10 - 1)/2$, no?}\colin{I think r(G)/20 works, added some details.}
    If it is~$G[C \cup N \cup F]$, then we can apply~\cref{thm:richvertex} to find a $(\delta_k/20)$-rich vertex,
    and if it is~$G[R']$, then we can apply the induction hypothesis to find an $(\eps_{k-1}/20)$-rich vertex.
    In either case, this gives an $\eps_k$-rich vertex.\cqed
  \end{proof}
  Thus we can now assume that $\card{A} \ge \frac{9}{10} r(G)$.

\smallskip

  Let~$B_1$, resp.~$B_2$, be the set of vertices of~$B$ incident to exactly one, resp.\ at least two edges of~$A$,
  and let~$A_1,A_2 \subseteq A$ be the set of edges of~$A$ incident to~$B_1,B_2$ respectively.
  Remark that~$A_1,A_2$ and~$B_1,B_2$ partition~$A$ and~$B$ respectively, and~$\card{A_1} = \card{B_1}$.
  \begin{claim}
    If~$\card{A_2} \ge \frac{4}{9} \card{A}$, then~$G$ has an $\eps_k$-rich vertex.
  \end{claim}
  \begin{proof}
    Assume that $\card{A_2} \ge \frac{4}{9} \card{A}$, and thus $\card{A_2} \ge \frac{2}{5} r(G)$.
    By \cref{lem:avg-degreeg11}, $G$ is $2k$-degenerate, hence it can be vertex-partitioned into $k+1$ forests.
    Consider this partition restricted to~$B_2$, and choose $B_3 \subseteq B_2$ which induces a forest and maximizes the set~$A_3\subseteq A_2$ of edges incident to~$B_3$.
    Thus~$\card{A_3} \ge \card{A_2} / (k+1) \ge \frac{2}{5(k+1)} r(G)$.
    By \cref{claim:roots-F}, $F \cup B_3$ is a forest,
    hence \cref{thm:richvertex} applies to~$G[C \cup N \cup F \cup B_3]$.

    By \cref{clm:CNF-connected}, $G[C \cup N \cup F]$ is connected,
    thus adding the vertices~$B_3$ and the edges~$A_3$ increases the cycle rank by $\card{A_3} - \card{B_3}$.
    This quantity is at least~$\card{A_3}/2$ since any vertex of~$B_3$ is incident to at least two edges of~$A_3$,
    and each edge of~$A_3$ is incident to exactly one vertex of~$B_3$.
    Thus \cref{thm:richvertex} yields the existence of a vertex of degree at least
    \begin{equation}
      \delta_k \cdot r(G[C \cup N \cup F \cup B_3])
      \ge \frac{\card{A_3}}{2} \delta_k
      \ge \frac{1}{5(k+1)} \delta_k \cdot r(G)
      \ge \eps_k \cdot r(G)
    \end{equation}
    as desired.\cqed
  \end{proof}
  Thus we can now assume that $\card{A_1} \ge \frac{5}{9} \card{A}$,
  and thus $\card{B} \ge \frac{5}{9} \card{A} \ge \frac{1}{2} r(G)$.

\smallskip

  Let~$X$, resp.~$Y$, be the set of vertices of~$B$ with degree at least~3, resp.\ exactly 2, in $G[R']$.
  By \cref{clm:min-degree-R}, this is a partition of~$B$.
  \begin{claim}
    If $\card{X} \ge \card{B}/5$, then~$G$ has an $\eps_k$-rich vertex.
  \end{claim}
  \begin{proof}
    Assume that $\card{X} \ge \card{B}/5$, and thus $\card{X} \ge \frac{1}{10} r(G)$.

    The cycle rank is lower-bounded by the following sum:
    \begin{equation}
      \label{eq:edge-vertex-degree}
      r(G[R']) \ge \card{E(G[R'])} - \card{R'} = \frac12\sum_{x \in R'} (d_{G[R']}(x) - 2).
    \end{equation}
    By \cref{clm:min-degree-R}, every term in the sum is non-negative,
    and each~$x \in X$ contributes by at least~$1/2$ to the sum.
    Thus $r(G[R']) \ge \card{X}/2 \ge \frac{1}{20} r(G)$,
    and the induction hypothesis applied to~$G[R']$ (which is $\mathcal{O}_{k-1}$-free)
    yields an $(\eps_{k-1} / 20)$-rich vertex, which is also $\eps_k$-rich.\cqed
  \end{proof}
  Thus we can now assume that $\card{Y} \ge \frac{4}{5} \card{B} \ge \frac{2}{5} r(G)$.

\smallskip

  Let~$Z$ be the set of vertices of~$R'$ that either are in~$Y$ or have degree at least~3 in $G[R']$.
  Remark that~$Z$ is exactly the set of vertices of~$R'$ with degree at least~3 in~$G$.
  In~$G[R']$, a~\emph{direct path} is a path whose endpoints are in~$Z$, and whose internal vertices are not in~$Z$.
  In particular, internal vertices of a direct path have degree~$2$.
  A direct path need not be induced, as its endpoints may be adjacent.
  As a degenerate case, we consider a cycle that contains a single vertex of~$Z$ to be a direct path whose two endpoints are equal.
  One can naturally construct a~multigraph~$G_Z$
  with vertex set~$Z$ and whose edges correspond to direct paths in~$G[R']$.
  Remark that vertices of~$Z$ have the same degree in~$G_Z$ and in~$G[R']$.

  Any $y \in Y$ has two neighbors~$x_1,x_2$ in~$G_Z$.
  In degenerate cases, it may be that $x_1 = x_2 \neq y$ (multi-edge in~$G_Z$),
  in which case~$G[R']$ contains a banana between~$y$ and~$x_1$,
  or that $x_1 = x_2 = y$ (loop in~$G_Z$),
  in which case there is a cycle~$C_y$ which is a connected component of~$G[R']$,
  and such that~$y$ is the only vertex of~$Z$ in~$C_y$.
  We partition $Y$ into~$Y_i,Y_e$ as follows: for~$y,x_1,x_2$ as above,
  if $x_1,x_2 \in Y$, then we place~$y$ in~$Y_i$,
  and otherwise ($x_1$ or~$x_2$ is in~$Z \setminus Y$) we place~$y$ in~$Y_e$.

%\textcolor{red}{A:I think this part is not very clear (at least to me at first reading). I would say that in the direct path, all interior vertices are in $R'\setminus Y$. Also, it would be nice to explain the loop case a bit better: Is it true that we only are in the "loop" case if the connected component in which $y$ is contains no other vertex of $Y$ and is a cycle?} \louis{Yes, that's what I understood. It was also a bit cryptic to me when I read it for the first time, but then it made sense by reading the proof below. Explaining when we get a loop might certainly help the reader, in any case.}
% \colin{Indeed. I think the definition should be a lot clearer like this. The contracted hypergraph should also clarify the claim below.}

  \begin{claim}
    If $\card{Y_e} \ge \frac{3}{4}\card{Y}$, then~$G$ has an $\eps_k$-rich vertex.
  \end{claim}
  \begin{proof}
    Assume $\card{Y_e} \ge \frac{3}{4}\card{Y}$, and thus $\card{Y_e} \ge \frac{3}{10} r(G)$.

    By definition, any vertex of~$Y_e$ is adjacent in~$G_Z$ to some vertex of~$Z \setminus Y$.
    Thus, using that $d_{G_Z}(z) = d_{G[R']}(z)$ for any~$z \in Z$, we obtain
    \begin{equation}
      \label{eq:sum-degrees-Z}
      \sum_{z \in Z \setminus Y} d_{G[R']}(z) \ge \card{Y_e}.
    \end{equation}
    Recall inequality~\eqref{eq:edge-vertex-degree} on cycle rank:
    \begin{equation}
      r(G[R']) \ge \frac12\sum_{x \in R'} (d_{G[R']}(x) - 2).
    \end{equation}
    By \cref{clm:min-degree-R}, the terms of this sum are non-negative.
    Thus, restricting it to~$Z \setminus Y$, we have
    \begin{equation}
      r(G[R']) \ge \frac12\sum_{z \in Z \setminus Y} (d_{G[R']}(z) - 2).
    \end{equation}
    By definition of~$Z$, vertices of~$Z \setminus Y$ have degree at least~$3$ in~$G[R']$.
    Thus, each term of the previous sum satisfies $d_{G[R']}(z) - 2 \ge d_{G[R']}(z)/3$.
    It follows using~\eqref{eq:sum-degrees-Z} that
    \begin{equation}
      r(G[R']) \ge  \frac12\sum_{z \in Z \setminus Y} \frac{d_{G[R']}(z)}{3}
      \ge \frac{\card{Y_e}}{6} \ge \frac{1}{20} r(G).
    \end{equation}
    Thus the induction hypothesis applied to~$G[R']$ (which is $\mathcal{O}_{k-1}$-free)
    yields an $(\eps_{k-1} / 20)$-rich vertex, which is also $\eps_k$-rich.\cqed
  \end{proof}
  Thus we can now assume that $\card{Y_i} \ge \frac{1}{4}\card{Y} \ge \frac{1}{10} r(G)$.

\smallskip

  We now consider the induced subgraph~$H$ of~$G[R']$ consisting of~$Y$, and direct paths joining vertices of~$Y$.
  Thus~$H$ has maximum degree~$2$, and since~$G[R']$ is $\mathcal{O}_{k-1}$-free,
  at most~$k-2$ components of~$H$ are cycles, the rest being paths.
  Remark that the endpoints of paths in~$H$ correspond exactly to~$Y_e$.
  Also, each connected component of~$H$ must contain at least one vertex of~$Y$.

  We perform the following cleaning operations in  order:
  \begin{itemize}
    \item \label{item:split} In each cycle of~$H$, pick an arbitrary vertex and delete it,
      so that all connected components are paths.
    \item \label{item:reduce} Iteratively delete a vertex of degree $0$ or $1$ which is not in~$Y$,
      so that the endpoints of paths are all in~$Y$.
    \item \label{item:isolated} Delete all isolated vertex.
  \end{itemize}
  Let~$H'$ be the subgraph of~$H$ obtained after these steps.

  \begin{claim}
    \label{clm:internal-vertices-H}
    All but~$3(k-2)$ vertices of~$Y_i$ are internal vertices of paths of~$H'$.
  \end{claim}
  \begin{proof}
    If~$y \in Y_i$ belongs to a path of~$H$, then it must be an internal vertex of this path,
    and the path is unaffected by the cleaning operations.
    Thus it suffices to prove that in each cycle of~$H$,
    at most~3 vertices of~$Y_i$ are deleted or become endpoints of paths during the clean up.

    Let~$C'$ be a cycle of~$H$. If~$C'$ contains no more than~2 vertices of~$Y_i$, there is nothing to prove.
    Remark in this case that~$C'$ is entirely deleted by the clean up.
    Otherwise, let~$x$ be the vertex deleted from~$H$ (which may be in~$Y_i$),
    and let~$y_1,y_2$ be the first vertices of~$Y_i$ strictly before and after~$x$ in the cyclic order of~$C'$.
    Since~$C'$ has at least~3 vertices of~$Y_i$, $x,y_1,y_2$ are all distinct.
    Then, it is clear that the cleaning operations transform~$C'$ into a path with endpoints~$y_1,y_2$,
    such that any~$y \in Y_i \cap C'$ distinct from~$x,y_1,y_2$ is an internal vertex of this path.\cqed
  \end{proof}

  We now add~$H'$ to~$F$, which yields a forest by~\cref{claim:roots-F}.
  Recall that vertices of~$Y$ are adjacent to~$N \cup F$, and all endpoints of paths of~$H'$ are in~$Y$.
  Thus, in~$G[C \cup N \cup F \cup H']$, every vertex of~$H'$ has degree at least~$2$,
  and vertices of~$Y_i$ in the interior of paths of~$H'$ have degree at least~$3$.
  Since~$G[C \cup N \cup F]$ is connected by \cref{clm:CNF-connected},
  the addition of~$H'$ does not change the number of connected components.
  Using \cref{clm:internal-vertices-H}, this implies that
  \begin{equation}
    r(G[C \cup N \cup F \cup H']) \ge \card{Y_i} - 3(k-2).
  \end{equation}
  We finally apply~\cref{thm:richvertex} to~$G[C \cup N \cup F \cup H']$
  to obtain a vertex with degree at least
  \[ \delta_k \cdot (\card{Y_i} - 3(k-2)). \]
  Since~$G$ contains vertices of degree at least~$2$,
  we can always assume that $\eps_k \cdot r(G) \ge 2$, and thus
  \begin{equation}
    \card{Y_i} \ge \frac{1}{10} \cdot 2 \eps_k^{-1} \ge \frac{1}{5} \cdot 30(k-2) = 6(k-2).
  \end{equation}
  It follows that $\card{Y_i} - 3(k-2) \ge \card{Y_i}/2$, and the previous argument yields
  a vertex of degree at least $\frac{\delta_k}{2} \card{Y_i} \ge \frac{\delta_k}{20} r(G)$,
  which is an $\eps_k$-rich vertex.
\end{proof}

\section{Cutting trees into stars and paths}\label{sec:cutting}

Recall the statement of Theorem \ref{thm:richvertex}: we start with an $\mathcal{O}_k$-free graph $G$ of large girth, and divide its vertex set into some shortest cycle $C$, its neighborhood $N$, and the remainder of the vertex set $R$ (including the second neighborhood $S$ of $C$). Moreover we assume that $R$ induces a forest. Since $C$ is a shortest cycle, it is not difficult to check that every vertex of $S$ must have exactly one neighbor in $N$. Moreover, up to reducing the graph under consideration, we can assume that all the leaves of $G[R]$ lie in $S$.

Our goal in this section will be to simplify $G[R]$ by only keeping a linear number (in $|S|$) of subdivided stars or paths with endpoints in $S$. To this end it will be convenient to leave $C$ aside and only consider $N$ and $R$ for now (or more precisely what remains of $R$ after the graph has been reduced). Curious readers are invited to have a quick look at Figures \ref{fig:startree}, \ref{fig:starpath1} and \ref{fig:starpath2} to have an idea of how the results of this section will be used in the proof of Theorem \ref{thm:richvertex}.

\medskip

The paragraphs above motivate the following definitions. A forest $H$ is said to be \emph{$(S\subseteq F)$-decorated} if $V(H)=F$, every leaf of $H$ lies in $S\subseteq F$, and every connected component of $F$ contains at least 2 vertices.   A graph $H$ is said to be \emph{$(N,S\subseteq F)$-divided} if its vertex set is partitioned into two sets $F$ and $N$, such that 
\begin{itemize}
    \item $F$ induces a forest and $N$ is an independent set, 
    \item the neighborhood of $N$ in $F$ is a subset $S\subseteq V(F)$ containing all the leaves of $F$, 
    \item each vertex of $S$ has a unique neighbor in $N$, and
    \item every connected component of $H[F]$ contains at least two vertices.
\end{itemize}    

Note that the second and fourth conditions imply that $H[F]$ is $(S\subseteq F)$-decorated.
It can be deduced from the definition that if $H$ is $(N,S\subseteq F)$-divided, then it does not contain $K_{3,3}$ as a subgraph.

\medskip

A \emph{subdivided star} is a graph with at least two vertices, which is a subdivision of a star (a graph obtained from a star by replacing its edges by paths of arbitrary length). We insist on the fact that we do not consider singleton vertices as subdivided stars.
A path (on at least two vertices) is a special case of subdivided star. The \emph{center} of a subdivided star is the vertex of degree at least $3$, if any. 
%The\emph{ center} 
If none, the subdivided star is a path, and its \emph{center} is a vertex of degree $2$ that belongs to $S$, if any, and an arbitrary vertex otherwise.  
We say that a forest $F'\subseteq F$ is \emph{$S'$-clean}, for some $S'\subseteq S$, if $V(F')\cap S' = L(F')$, where $L(F')$ denotes the set of leaves of~$F'$. 
We define being \emph{quasi-$S'$-clean} for a subdivided star as intersecting $S'$ at exactly its set of leaves, plus possibly its center.
Formally, a subdivided star $T$ is \emph{quasi-$S'$-clean} if $L(T) \subseteq V(T)\cap S' \subseteq L(T)\cup \{c\}$ where $c$ is the center of $T$. 
The \emph{degree} of a subdivided star is the degree of its center. A forest $F'\subseteq F$ of subdivided stars is \emph{quasi-$S'$-clean}, for some $S'\subseteq S$, if all its connected components are quasi-$S'$-clean (subdivided stars).

\begin{figure}[htb]
 \centering
 \includegraphics[scale=1]{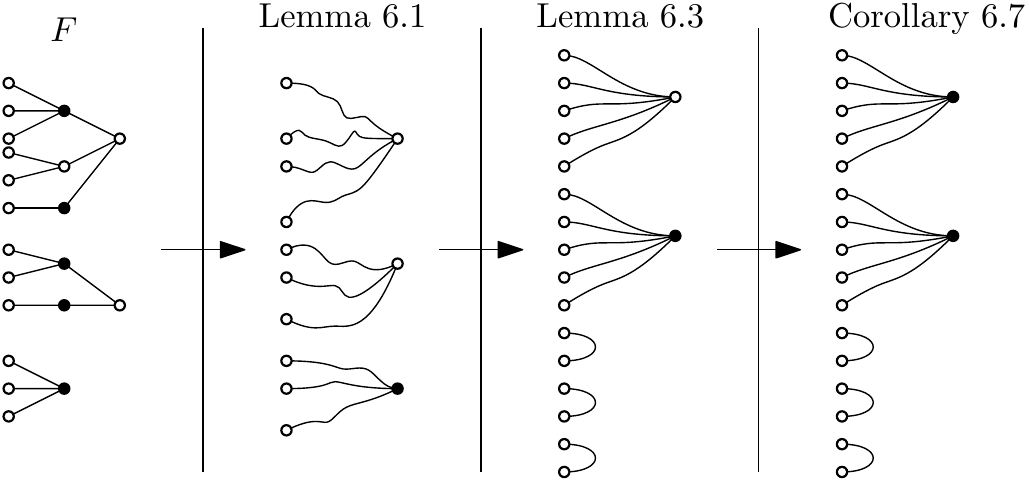}
 \caption{A visual summary of Section \ref{sec:cutting}.}
 \label{fig:cutting}
\end{figure}

\medskip

Our approach in this section is summarized in Figure \ref{fig:cutting}. We start with our forest $F$ and a subset $S$ of vertices including all the leaves of $F$ (the vertices of $S$ are depicted in white, while the vertices of $F-S$ are depicted in black). We first extract quasi-$S$-clean  subdivided stars  (Lemma \ref{lemma:stars}). We then extract quasi-$S$-clean subdivided stars of large degree, or $S$-clean paths (Lemma \ref{lem:throwawaysmallstars}). Finally we extract $S$-clean subdivided stars of large degree or paths (Corollary \ref{cor:dec_tree}). At each  step the number of vertices of $S$ involved in the induced subgraph of $F$ we consider is linear in $|S|$.

\begin{lemma} \label{lemma:stars}
 %For any integer $D \geq 2$, 
 Let $H$ be an $(S\subseteq F)$-decorated forest. Then there is a subset $F^*\subseteq F$ containing at least $\tfrac1{2} |S|$ vertices of $S$ such that each connected component of $H[F^*]$ is a quasi-$S$-clean subdivided star.% one of the following.
% \begin{itemize}
%  \item A quasi-$S$-clean subdivided star.% whose center has degree at least $D$.
%  \item An $S$-clean path.
% \end{itemize}
%Moreover, if $S=L(F)$, all subdivided stars in $H[F^*]$ are $S$-clean. %Marthe: it's weird to say it as part of the lemma no? 
%Louis. We're using this specific property later (several times if I recall correctly) so I preferred to state it here explicitly (and add a sentence in the proof), rather than say later that when $S=L(F)$ we have the stronger property (but I suppose it's a matter of taste! Do as you feel is best). 
\end{lemma}

\begin{proof}
We first use the following claim. 
\begin{claim}\label{cl:deleteedges}
 There is a set of edges $X \subseteq E(H)$ such that every connected component of $H\setminus X$ is either a quasi-$S$-clean subdivided star or a single vertex that does not belong to $S$.
\end{claim}

\begin{proof}
We proceed greedily, starting with $X=\emptyset$. While $H\setminus X$ contains a component $T$ and an edge $e\in T$ such that each of the two components of $T-e$ contains either no vertex of $S$ or at least two vertices of $S$, we add $e$ to $X$.

Observe that in $H$, every connected component  contains at least $2$ vertices of $S$. Throughout the process of defining $X$, every connected component of $H\setminus X$  contains either $0$ or at least $2$ vertices of $S$.

At the end of the process, for any connected component $T$ of $H\setminus X$ with at least one edge, all the leaves of $T$ belong to $S$. Otherwise, the edge incident to the leaf of $T$ that is not in $S$ can be added to $X$.

Thus, $H\setminus X$ does not contain any component with more than one vertex of degree at least 3, since otherwise any edge on the path between these two vertices would have been added to $X$, yielding two components containing at least 2 leaves, and thus at least 2 vertices of $S$. 

Observe also that if $H\setminus X$ contains a component $T$ with a vertex $v\in S$ that has degree 2 in $T$, then $T$ is a path containing exactly 3 vertices of $S$, and thus $T$ is a subdivided star whose center and leaves are in $S$, and whose other internal vertices are not in $S$.\cqed
\end{proof}

To conclude, we need to select connected components of $H\setminus X$ with at least two vertices of $S$ and that are pairwise independent in $H$. Consider the minor $G_H$ of $H$ obtained by contracting each connected component of $H\setminus X$ into a single vertex and deleting those that are a single vertex not in $S$. Since $H$ is a forest, the graph $G_H$ is a forest. We weigh each vertex of $G_H$ by the number of elements of $S$ that the corresponding connected component of $H\setminus X$ contains. Since $G_H$ is a forest, there is an independent set $\{u_1,u_2,\ldots,u_p\}$ that contains at least half the total weight. The connected components corresponding to $u_1,u_2,\ldots,u_p$ together form a forest $H[F^*]$ with the required properties.
%Finally, we note that if $S=L(F)$, then the elements of $S$ have degree 1 in $F$ and thus cannot be centers of subdivided stars. It follows that all quasi-$S$-clean subdivided stars are indeed $S$-quasi-clean, as desired.
\end{proof}

We observe that subdivided stars of small degree can be transformed into paths for a low price, as follows. A~\emph{subdivided star forest}  is a forest whose components are subdivided stars (possibly paths). % We say that a star-path forest $F'\subseteq F$ is \emph{quasi-$S'$-clean} for some $S'\subseteq S$ if each component of $H[F']$ is an $S'$-clean path or a quasi-$S'$-clean subdivided star.

\begin{lemma}\label{lem:throwawaysmallstars}
 Let $H$ be an $(S\subseteq F)$-decorated forest. For every $S'\subseteq S$, every quasi-$S'$-clean subdivided star forest $F' \subseteq F$, and every integer \mbox{$D \geq 2$}, there is a subdivided star forest $F'' \subseteq F'$ such that every connected component of $H[F'']$ is either an $S'$-clean path or a quasi-$S'$-clean subdivided star of degree at least $D$.  Additionally, $F''$ contains at least $\frac{2|S'\cap F'|}{D}$ vertices of $S'$.
 %%Moreover, if $F'$ is $S'$-clean, then the subdivided stars in $F''$ are $S'$-clean.
\end{lemma}

\begin{proof}
We define $F''$ from $F'$ as follows.
Consider a connected component $T$ of $H[F']$. %, the tree $T$ is either an $S$-clean path or a quasi-$S$-clean subdivided star. In the first case, we add $T$ to $F''$. In the second case, i
If the center of $T$ has degree at least $D$, we add $T$ to $F''$. Consider now the case where $T$ is a quasi-$S'$-clean subdivided star whose center $c$ has degree less than $D$. If $c \in S'$, we select a non-edgeless path $P \subseteq T$ between $c$ and $S'$, and add $P$ to $F''$. If $c \not\in S'$, we select two internally-disjoint paths $P_1, P_2 \subseteq T$ between $c$ and $S'$, and add $P_1\cup P_2$ to $F''$. Note that $P_1\cup P_2$ yields an $S'$-clean path. 
%Observe that if $F'$ is an $S'$-clean forest, then $F''$ is also an $S'$-clean forest.

To see that $F''$ contains at least $\frac{2|S'\cap F'|}{D}$ vertices of $S'$, we simply observe that in the second case, out of a maximum of $(D-1)+1$ vertices of $S'$ in a component $T$, we keep at least $2$ in $F'$. This adds up to $\frac{2|S'|}{D}$ vertices of $S'$ since connected components of $H[F']$ are disjoint by definition.
\end{proof}

\begin{lemma}\label{lem:cleansmallstars}
 Let $H$ be an $\mathcal{O}_k$-free graph which is $(N,S\subseteq F)$-divided. If each vertex of $N$ has degree less than $\tfrac1{8k} |S|$, then one of the following holds.
  \begin{itemize}
      \item there is a subset $S'$ of $S$ and a subset $F_2$ of $F$ such that $F_2$ contains $\frac1{32}|S|$ vertices of $S'$, and each connected component of $H[F_2]$ is an $S'$-clean subdivided star. 
      \item there is a subset $F_3$ of $F$ such that every connected component of $F_3$ is a quasi-$S$-clean subdivided star of degree at most $4$ and $F_3$ contains at least $\tfrac18|S|$ vertices of $S$.
  \end{itemize}
\end{lemma}

\begin{proof}
Let $F^*\subseteq F$ be the forest obtained from Lemma~\ref{lemma:stars}, applied to the $(S\subseteq F)$-decorated forest $H[F]$. Then $F^*$  contains at least $\tfrac12|S|$ vertices of~$S$, and each component of $H[F^*]$ is a quasi-$S$-clean subdivided star or an $S$-clean path. We define the \emph{label} of a vertex of $S$ to be its only neighbor in $N$. 

\begin{claim}\label{cl:labels}
  %Let $\varepsilon>0$. If the vertices of $N$ all have degree less than $\varepsilon |S|$, t
  There is a subset $F_1$ of $F^*$ containing at least $\tfrac14|S|$ vertices of $S$, such that no subdivided star of $F_1$ has its center and one of its endpoints sharing the same label.
\end{claim}

\begin{proof}
Let $\ell$ be the maximum integer such that there exist $\ell$ subdivided stars 
 $S_1,S_2,\ldots,S_\ell$ in $H[F^*]$ and $\ell$ different labels $v_1,\dots , v_\ell\in N$, such that for any $1\le i \le \ell$, $S_i$ has its center and at least one of its endpoint labeled $v_i$. Note that in this case $G$ contains  $\ell$ independent cycles, and thus $\ell<k$ by assumption.
 
For any $1\le i \le \ell$, remove all the leaves $u$ of $F^*$ that are labeled $v_i$, and also remove the maximal path of $H[F^*]$ ending in $u$. By assumption, there are at most $\tfrac1{8k} |S|$ such vertices $u$ for each $1\le i \le \ell$, and thus we delete at most $k\cdot \tfrac1{8k}|S|\le \tfrac18|S|$ vertices of $S$ from $F^*$. We also delete the centers that have no leaves left (there are at most $k\cdot \tfrac1{8k}|S|\le \tfrac18|S|$ such deleted centers). Let $F_1$ be the resulting subset of $F^*$. Note that $F_1$ contains at least $|F^*\cap S|-2\cdot\tfrac18|S|\ge  (\tfrac12-\tfrac14) |S|=\tfrac14|S|$ vertices of $S$. \cqed
\end{proof}

We can assume that a subset $Y$ of at least $\tfrac18|S|$ vertices of $S$ in the forest $F_1$ obtained from Claim~\ref{cl:labels} are involved in a quasi-$S$-clean subdivided star of degree at least $5$. Indeed, otherwise at least $\tfrac18|S|$ vertices of $S$ in the forest $F_1$ obtained from Claim~\ref{cl:labels} are involved in a quasi-$S$-clean subdivided star of degree at most $4$ (note that an $S$-clean path is an $S$-clean subdivided star), and in this case the second outcome of \cref{lem:cleansmallstars} holds.

For each label $v\in N$, we choose uniformly at random with probability $\tfrac12$ whether $v$ is a \emph{center label} or a \emph{leaf label}. We then delete all the subdivided stars of $F_1$ whose center is labeled with a leaf label, and all the leaves whose label is a center label. Moreover, we delete from $N$ all the vertices that are a center label, and let $S'$ be the set of vertices of $S$ whose neighbor in $N$ is not deleted. %(and call $N'$ the resulting subset of $N$ and $S'$ the set of vertices of $S$ that have a neighbor in $N'$).
 
Take a vertex $u$ of $Y$. If $u$ is a center of a subdivided star, then the probability that $u$ is not deleted is at least $\tfrac{1}{2}$. If $u$ is a leaf, $u$ is kept only if $u$ and the center of the subdivided star it belongs to (which has by construction a different label) are correctly labeled, so $u$ is kept with probability at least $\tfrac{1}{4}$. Overall, each vertex $u$ of $Y$ has probability at least $\tfrac{1}{4}$ to be kept. Thus the expectation of the fraction of vertices of $Y$ not deleted is at least $\tfrac{1}{4}$, thus we can find an assignment of the labels to leaf labels or center labels, such that a subset $Z \subseteq Y$ with $|Z|\geq \frac14 |Y|$ survives.

We then iteratively delete vertices of degree $1$ that do not belong to $S'$ and all vertices of degree $0$. Let $F_2$ be the resulting forest. Note that $S'$ contains only the endpoints of stars with a leaf label, thus the forest $F_2$ is $S'$-clean. It remains to argue that $F_2$ contains a significant fraction of vertices of $S$. Note that a connected component of $F_1$ is deleted if and only if it contains at most one element of $Z$. Every such component has at least $4$ elements in $Y \setminus Z$, hence there are at most $\frac14 \cdot \frac34 |Y|=\tfrac3{16}|Y|$ such components. It follows that $F_2$ contains at least $|Z|-\frac{3}{16} |Y|\ge \tfrac14|Y|-\tfrac3{16}|Y|\ge \tfrac{1}{16}|S|$ elements of $Z \subseteq S$. 
\end{proof}

We now have all the ingredients to obtain the following two corollaries.

\begin{corollary}
\label{cor:quasicleanforest}
 Let $H$ be an $(S\subseteq F)$-decorated forest. For any $D \geq 2$, there is a subset $F^*\subseteq F$ containing at least $\tfrac1{2D} |S|$ vertices of $S$ such that each 
  \begin{enumerate}
      \item $F^*$ induces a quasi-$S$-clean subdivided star forest whose components all have degree at least $D$, or
      \item $F^*$ induces an $S$-clean path forest.
  \end{enumerate}
\end{corollary}

Corollary~\ref{cor:quasicleanforest} follows from Lemma~\ref{lemma:stars} by applying Lemma~\ref{lem:throwawaysmallstars} and observing that one of the two outcomes contains half the corresponding vertices in $S$.

\begin{corollary}\label{cor:dec_tree}
Let $H$ be an $\mathcal{O}_k$-free graph which is $(N,S\subseteq F)$-divided, and let  $D\geq 2$.
If each vertex of $N$ has degree less than $\tfrac1{8k} |S|$, then there are $F''\subseteq F$, $S'\subseteq S$ such that $F''$ contains at least $\tfrac{1}{32D} |S|$ vertices of $S'$ and one of the following two cases apply.
  \begin{enumerate}
      \item $F''$ induces an $S'$-clean subdivided star forest whose components all have degree at least $D$, or
      \item $F''$ induces an $S'$-clean path forest.
  \end{enumerate}
\end{corollary}

Similarly, Corollary~\ref{cor:dec_tree} follows from Lemma~\ref{lem:cleansmallstars} by applying Lemma~\ref{lem:throwawaysmallstars} and observing that one of the two outcomes contains half the corresponding vertices in $S$.

\section{Trees, stars, and paths}\label{sec:tsp_new}

In the proof of Theorem \ref{thm:richvertex}, we will apply Corollaries \ref{cor:quasicleanforest} and \ref{cor:dec_tree} several times, and divide our graph into two parts: a union of subdivided stars on one side, and a union of subdivided stars or paths on the other side (see again Figures \ref{fig:startree}, \ref{fig:starpath1} and \ref{fig:starpath2} for an idea of what these two sides will correspond to in  the final applications). We now explain how to find a rich vertex in this context.

\medskip

We start with the case where subdivided stars appear on both sides.

\begin{lemma}[Star-star lemma]\label{lemma:starstar_new}Let $c>0$ be the constant of Lemma \ref{lem:min-degree}.
Let $H$ be an $\mathcal{O}_k$-free graph whose vertex set is the union of two sets $L,R$, such that 
\begin{itemize}
    \item $S=L\cap R$ is an independent set,
    \item there are no edges between $L\setminus S$ and $R\setminus S$, and
    \item $L$ (resp.\ $R$) induces in $H$ a disjoint union of subdivided stars, whose centers have average degree at least $3c k \log k$, and whose set of leaves is precisely $S$.
\end{itemize} 
Then $H$ contains a vertex of degree at least $\tfrac1{2f'(3,k)}|S|=\Omega(\tfrac1{k^3}|S|)$, where $f'$ is the function of Corollary \ref{corollary:banana}.
\end{lemma}
\begin{proof}
Note that $H$ does not contain $K_{3,3}$ as a subgraph (but might contain $K_{2,2}$ as a subgraph) and is $\mathcal{O}_k$-free. By~\cref{corollary:banana}, there is a set $X$ of at most $f'(3,k)$ vertices of $H$ such that all bananas of $H$ intersect $X$. Since the centers of the subdivided stars are the only vertices of degree larger than~2 in $H$, we can assume that $X$ is a subset of the centers of the subdivided stars. 

Assume first that less than $\tfrac12|S|$ vertices of $S$ are leaves of subdivided stars centered in an element of $X$.
 Let $S'\subseteq S$ be the leaves of the subdivided stars whose center is not in $X$ (note that $|S'|\ge \tfrac12 |S|)$, and remove from the subdivided stars of $H[L]$ and $H[R]$ all branches whose endpoint is not in $S'$ to get new sets of vertices $L', R'$. The centers of the resulting $S'$-clean subdivided stars  now have average degree at least $\tfrac12\cdot 3c k \log k> ck\log k$. We denote the resulting $S'$-clean subdivided stars of $H[L']$ by $S_1, S_2$, etc. and their centers by $s_1, s_2$, etc. Similarly, we denote the resulting $S'$-clean subdivided stars of $H[R']$ by $S_1',S_2'$, etc. and their centers by $s_1',s_2'$, etc. 
Observe that by the definition of $X$, for any two centers $s_i,s_j'$, there is at most one vertex $u\in S'$ which is a common leaf of $S_i$ and $S_j'$.  

Let $B$ be the bipartite graph with partite set $s_1, s_2, \ldots$ and $s_1', s_2', \ldots$, with an edge between $s_i$ and $s_j'$ if and only if some vertex of $S'$ is a common leaf of $S_i$ and $S_j'$. Note that $B$ has average degree more than $c k \log k$, and some induced subgraph of $H$ (which is $\mathcal{O}_k$-free) contains a strict subdivision of $B$. 
This contradicts \cref{lem:min-degree}.

So we can assume that at least $\tfrac12|S|$ vertices of $S$ are leaves of subdivided stars centered in an element of $X$. Then some vertex of $X$ has degree at least $\tfrac1{2f'(3,k)}|S|$, as desired.
%
%(feel free to rearrange the proof as you wrote it earlier)
%Let $X=\{x_1, \dots, x_r\}$ be the centers of the stars in $L\cup S$ and $Y=\{y_1, \dots, y_t\}$ be the centers of the stars in $R \cup S$ and we assume that each center has degree at most $\eps |S|$.
%
%We consider the bipartite graph $H_B$ on $X \cup Y$, where $x_i$ is adjacent to $y_j$ if they both have $s\in S$ as a leaf. We say a pair $(x_i,y_j)$ is cyclic if $x_i,y_j$ share at least two leaves. All cyclic pairs can be covered by at most $k-1$ vertices in $X\cup Y$ because $H$ is $\mathcal{O}_k$-free and suppose $S'\subset S$ are the neighbors in $S$ of those $k-1$ vertices. Note that $|S'|\leq (k-1)\eps|S|$. 
%
%Note that $D\leq \frac{1}{|X\cup Y|}\sum_{v \in X \cup Y} \deg(v)=|S|$. Hence the average degree of the bipartite graph $H_B$ is at least $(1-(k-1)\eps)D$, since the vertices $s\in S\setminus S'$ create distinct edges in $H_B$. Moreover, $H$ contains a strict subdivision of $H_B$. Since  $(1-(k-1)\eps)D>cklogk$ where $c$ is as in Lemma \ref{lem:min-degree}, $H$ is not $\mathcal{O}_k$-free, a contradiction.
%This is not true since vertices in S can have neighbors: By the path multiplicity Lemma(\textcolor{red}{If we assume that we deleted multiple paths}), $x_i$ and $y_j$ share at most two leaves. Hence the bipartite graph has minimum degree at least $D/2$. $k$ disjoint cycles in the bipartite graph induce $k$ disjoint cycles in $H$ since centers in $X$ share no leaves and centers in $Y$ share no leaves. 
\end{proof}

We now consider the case where subdivided stars appear on one side, and paths on the other.

\begin{lemma}[Star-path lemma]\label{lemma:starpath_new}
Let $c>0$ be the constant of~\cref{lem:min-degree}. Let $H$ be an $\mathcal{O}_k$-free  graph whose vertex set is the union of two sets $L,R$, such that 
\begin{itemize}
    \item $S=L\cap R$ is an independent set,
    \item there are no edges between $L\setminus S$ and $R\setminus S$,
    \item $L$  induces in $H$ a disjoint union of paths, whose set of endpoints is precisely $S$, and
    \item $R$ induces in $H$ a disjoint union of subdivided stars, whose centers have average degree at least $4c k \log k$, and whose set of leaves is precisely $S$.
\end{itemize} 
Then $H$ contains a vertex of degree at least $\tfrac1{3f'(2,k)} |S|=\Omega(\tfrac1{k^2}|S|)$, where $f'$ is the function of~\cref{corollary:banana}.
\end{lemma}

\begin{proof}
Note that $H$ does not contain $K_{2,2}$ as a subgraph,  and is $\mathcal{O}_k$-free. By~\cref{corollary:banana}, there is a set $X$ of at most $f'(2,k)$ vertices of $H$ such that all bananas of $H$ intersect $X$. Since the centers of the subdivided stars are the only vertices of degree more than 2 in $H$, we can assume that $X$ is a subset of the centers of the subdivided stars. 

Assume first that less than $\tfrac13|S|$ vertices of $S$ are leaves of subdivided stars centered in an element of $X$.
Then there are at least $\tfrac16|S|$ paths in $H[L]$ whose endpoints are not leaves of stars centered in $X$. Let $S'\subseteq S$ be the endpoints of these paths (note that $|S'|\ge \tfrac13 |S|$), and remove from the subdivided stars of $H[R]$ all branches whose endpoint is not in $S'$ to get $R'$. The centers of the resulting $S'$-clean subdivided stars in $H[R']$ now have average degree at least $\tfrac13\cdot 4c k \log k> ck\log k$. We denote these subdivided stars by $S_1,\ldots,S_t$, and their centers by $s_1,\ldots,s_t$.

Given two centers $s_i,s_j$, we say that a pair $u_i,u_j\in S'$ is an \emph{$\{i,j\}$-route} if $u_i$ is a leaf of $S_i$, $u_j$ is a leaf of $S_j$, and there is a path with endpoints $u_i,u_j$ in $H[L]$. Observe that by the definition of $X$, for every pair $s_i,s_j$, there is at most one $\{i,j\}$-route. 

Let $G$ be the graph with vertex set $s_1,\ldots,s_t$, with an edge between $s_i$ and $s_j$ if and only if there is an $\{i,j\}$-route. Note that $G$ has average degree more than $c k \log k$, and some induced subgraph of $H$ (which is $\mathcal{O}_k$-free) contains a strict subdivision of $G$. This contradicts~\cref{lem:min-degree}. 

So we can assume that at least $\tfrac13|S|$ vertices of $S$ are leaves of subdivided stars centered in an element of $X$. Then some vertex of $X$ has degree at least $\tfrac1{3f'(2,k)}|S|$, as desired.
\end{proof}

From the two previous lemmas and~\cref{lemma:stars} we deduce the following.

\begin{lemma}[Star-tree lemma]\label{lemma:startree_new}
There is a constant $c>0$ such that the following holds. 
Let $H$ be an $\mathcal{O}_k$-free graph which does not contain $K_{t,t}$ as a subgraph. Assume that the vertex set of $H$ is the union of two sets $L,R$, such that 
\begin{itemize}
    \item $S=L\cap R$ is an independent set partitioned into $S_P, S_T$,
    \item there are no edges between $L \setminus S$ and $R\setminus S$,
    \item $L$ induces in $H$ a disjoint union of subdivided stars, whose centers have average degree at least $(8ck \log k)^2$, and whose set of leaves is equal to $S$, and
    \item $R$ induces in $H$ the disjoint union of 
    \begin{itemize}
    \item paths on a vertex set $R_P$, whose set of endpoints is equal to $S_P$, and 
    \item a tree $T$ on a vertex set $R_T$ such that $S_T$ is a subset of leaves of $T$.
    \end{itemize}
\end{itemize} 
Then $H$ contains a vertex of degree at least $\Omega(\tfrac1{k^4\log k} |S|)$. 
\end{lemma}
\begin{proof} 
Let $c>0$ be the constant of~\cref{lem:min-degree}. Assume first that $|S_T|\le 1$. Then since the subdivided stars of $L$ have average degree at least $(8ck \log k)^2$, we have $|S_P|=|S|-|S_T|\ge (8ck \log k)^2-1\ge 1$ and thus $|S_P|\ge \tfrac12|S|$. By removing the branch of a subdivided star of $L$ that has an endpoint in $S_T$ (if any), we obtain a set of $S_P$-clean subdivided stars of average degree at least $\tfrac12\cdot (8ck \log k)^2\ge 4ck \log k$. By~\cref{lemma:starpath_new}, we get a vertex of degree at least $\Omega(\tfrac1{k^2}|S_P|)=\Omega(\tfrac1{k^2}(|S|))$, as desired. So in the remainder we can assume that $|S_T|\ge 2$.

Let $T'$ be the subtree of $T$ obtained by repeatedly removing leaves that are not in $S_T$. Since $|S_T|\ge 2$,  $L(T')=S_T$.
Observe that $F'=T'\cup R_P$ is an $S$-clean forest (with $L(F')=S$), thus any $S$-quasi-clean subforest of $F'$ is $S$-clean.
It follows from \cref{cor:quasicleanforest} (applied to $S$, $F'$, and $D=4ck \log k$) that $F'$ contains a subset $F^*$ containing at least $\tfrac1{2\cdot 4ck \log k}|S|$ vertices of~$S$, such that $H[F^*]$ induces either (1) an $S$-clean forest of path, or (2) an $S$-clean forest of subdivided stars of degree at least $4ck\log k$.

We denote this intersection of $S$ and $F^*$ by $S^*$, and we remove in the subdivided stars of $H[L]$ all branches whose endpoint is not in $S^*$ to get a new set of vertices $L^*\subset L$. By assumption, the  average degree of the subdivided stars in $L^*$ is at least $\tfrac{(8ck\log k)^2}{8ck\log k} = 8ck\log k \ge 4ck\log k$. 

In case (1) above we can now apply \cref{lemma:starpath_new}, and in case (2) we can apply \cref{lemma:starstar_new}. In both cases we obtain a vertex of degree at least $\Omega(\tfrac1{k^3}|S^*|)=\Omega(\tfrac1{k^4\log k}|S|)$, as desired.
\end{proof}

\section{Proof of Theorem \ref{thm:richvertex}}\label{sec:main}
We start with recalling the setting of \cref{thm:richvertex}. The graph $G$ is a connected $\mathcal{O}_k$-free graph of girth at least 11, and $C$ is a shortest cycle in $G$. The neighborhood of $C$ is denoted by $N$, and the vertex set $V(G)\setminus(C\cup N)$ is denoted by $R$. The subset of $R$ consisting of the vertices adjacent to $N$ is denoted by $S$. Since $C$ is a shortest cycle, of size at least 11, each vertex of $S$ has a unique neighbor in $N$, and a unique vertex at distance 2 in $C$. Moreover $N$ and $S$ are independent sets. In the setting of \cref{thm:richvertex}, $R$ is a forest.

\smallskip

Our goal is to prove that there is a vertex whose degree is linear in the cycle rank $r(G)$.
To this end, we assume that $G$ has maximum degree at most $\delta \cdot r(G)$,
for some $\delta > 0$, and prove that this yields a contradiction if $\delta$ is a small enough function of $k$.

\smallskip

By \cref{lem:cycle-rank}, we can assume that $G$ is reduced, i.e., contains no vertex of degree $0$ or $1$.
If $G$ consists only of the cycle $C$, then $r(G) = 1$ and the theorem is immediate.
Thus we can assume that $N$ is non-empty, which in turn implies that $S$ is non-empty since $G$ is reduced. Since $R$ does not contain any vertex of degree 0 or 1 in $G$, we also have that  $G[R]$ does not contain any isolated vertex (all its components have size at least 2) and all the leaves of  $G[R]$ lie in $S$. Using the terminology introduced in Section \ref{sec:cutting}, $G[R]$ is an $(S\subseteq R)$-decorated forest, and $G\setminus V(C)$ is $(N,S\subseteq R)$-divided.

\medskip

Using that $G$ is connected, remark that
\begin{equation}
  \label{eq:cycle-rank-connected}
  r(G) = \card{E(G)} - \card{V(G)} + 1 = 1 + \frac{1}{2} \sum_{v \in V(G)} (d(v) - 2).
\end{equation}
We start with proving that the cardinality of $S$ is at least the cycle rank $r(G)$.

\begin{claim}\label{claim:Slinear}
$|S|\ge r(G)$, and thus $G$ has maximum degree at most $\delta |S|$.
\end{claim}
%\textcolor{red}{A: changed the constant in the claim}
\begin{proof}
  Observe that $\tfrac12\sum_{v\in C \cup N}(d(v)-2)=\frac{1}{2}|S|$.
  Furthermore $\tfrac12\sum_{v\in R}(d(v)-2)$ is equal to $\tfrac12|S|$ minus the number of connected components of $G[R]$, as $R$ induces a forest and each vertex of $S$ has a unique neighbor outside of $R$.
  Since $R$ is non-empty, it follows from~\eqref{eq:cycle-rank-connected} that $r(G)\le |S|$.
  We assumed that $G$ has maximum degree at most $\delta \cdot r(G)$ which is at most $\delta |S|$, as desired.\cqed
\end{proof}

 %If $|N|\le 10$, then some vertex $v\in N$ has at least $\tfrac1{10}|S|$ neighbors in $S$, which contradicts Claim \ref{claim:Slinear} for any $\delta < \tfrac1{10}$. Thus we can assume in the remainder of the proof that $|N|\ge 11$. \louis{Just in case.}

%\medskip

%Next, we prove that the removal of a a specific vertex of $C$ does not affect the cycle rank of $G$ by a large factor.

%\begin{claim}\label{claim:removalC}
%There is a vertex $v\in C$ such that $r(G-v)\ge  r(G)-\tfrac15|S|$
%\end{claim}

%\begin{proof}
%Observe  that for any vertex $v$ in $G$, $r(G)\le r(G-v)+d(v)-1$. Since $|C|\ge 11$ and vertices of $C$ have disjoint neighborhoods in $N$, there is a vertex $v\in C$ such that $d(v)\le 2+ \tfrac1{11}|N|\le 1+\tfrac1{5}|N|$ (where we have used $|N|\ge 11$ in the second inequality). Since $|N|\le |S|$, $v$ has degree at most $1+\tfrac1{5}|S|$ and thus $r(G)\le r(G-v)+d(v)-1\le r(G-v)+\tfrac1{5}|S|$.
%\end{proof}

In the remainder of the proof, we let $c>0$ be a sufficiently large constant such that~\cref{lem:cycle-handles,lemma:startree_new} both hold for this constant.

\medskip

We consider $\delta< \frac1{8k}$, and use Claim~\ref{claim:Slinear} to apply Corollary \ref{cor:dec_tree} to the subgraph $H$ of $G$ induced by $N$ and $F=R$ (which is $\mathcal{O}_k$-free), with $D=2\cdot (8ck\log k)^2$. 
We obtain subsets $N'\subseteq N$, $R''\subseteq R$ such that if we define $S'$ as the subset of $S\cap R''$ with a neighbor in $N'$,  we have $| S'|\ge \tfrac1{32D} |S|$ and at least one of the following two cases apply.
  
  \begin{enumerate}
      \item  Each connected component of $H[R'']$ is an $S'$-clean subdivided star of degree at least $D$, or
      \item Each connected component of $H[R'']$ is an $S'$-clean path.
  \end{enumerate}
 
 We first argue that the second scenario holds. 
  \begin{claim}\label{cl:case2}
   Each connected component of $H[R'']$ is an $S'$-clean path.
  \end{claim}

\begin{proof}

Assume for a contradiction that Case 2 does not apply, hence Case 1 applies.

\begin{figure}[htb]
 \centering
 \includegraphics[scale=0.9]{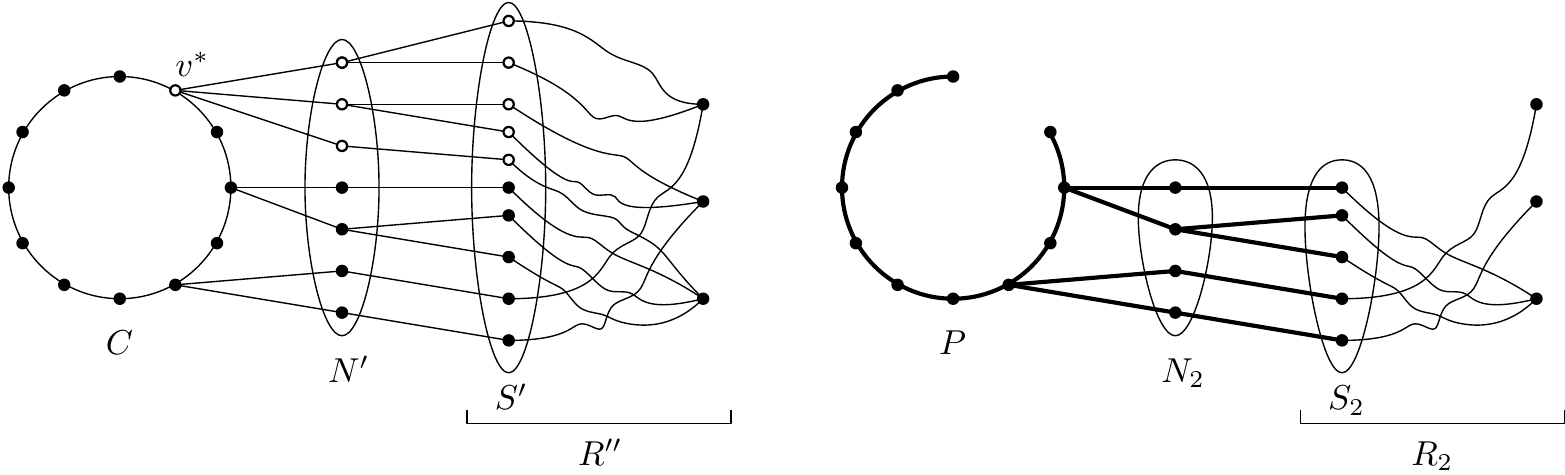}
 \caption{The graphs $G_1$ (left) and $G_2$ (right) in the proof of~\cref{cl:case2}.}
 \label{fig:startree}
\end{figure}

\medskip

Let $G_1$ be the subgraph of $G$ induced by $C\cup N'\cup R''$ (see Figure \ref{fig:startree}, left).
Since $|C|\ge 11$ and vertices of $C$ have disjoint second neighborhoods in $S'$, there exists a vertex $v^*\in C$ that sees at most $\frac{1}{11}|S'|$ vertices of $S'$ in its second neighborhood. If we remove from $G_1$ the vertex $v^*$, its neighborhood $N(v^*)\subseteq N'$, its second neighborhood $N^2(v^*)\subseteq S'$, and the corresponding branches of the subdivided stars of $R''$, we obtain a graph $G_2$ whose vertex set is partitioned into a path $P=C-v^*$, its neighborhood $N_2=N'-N(v^*)$, and the rest of the vertices $R_2$ (which includes the set $S_2=S'-N^2(v^*)$), with the property that each component of $G_2[R_2]$ is an $S_2$-clean subdivided star (see Figure \ref{fig:startree}, right). More importantly, \[|S_2|\ge \tfrac{10}{11}|S'|\ge \tfrac{10}{11}\cdot \tfrac1{32D} |S|\ge\tfrac1{36D} |S|,\]  and the average degree of the centers of the subdivided stars is at least $\tfrac{10}{11}D\ge (8ck \log k)^2$.

%L

\medskip

Observe that $P\cup N_2\cup S_2$ induces a tree in $G_2$, such that all leaves of $G_2[P\cup N_2\cup S_2]$ except at most two (the two neighbors of $v^*$ on $C$) lie in $S_2$, and non leaves of the tree are not in $S_2$.
We can now apply~\cref{lemma:startree_new}
with $R= P\cup N_2\cup S_2$ and $L=R_2$. It follows that $G_2$ contains a vertex of degree at least $\Omega(\tfrac1{k^4\log k} |S_2|)= \Omega(\tfrac1{k^6\log^3 k} |S|) >\delta |S|$. Since $G_2$ is an induced subgraph of $G$, this contradicts Claim \ref{claim:Slinear}.\cqed
\end{proof}

\medskip

%We can thus assume that Case 2 applies. That is, each connected component of $H[R'']$ is an $S'$-clean path. 
We denote the connected components of $H[R'']$ by $P_1, \ldots,P_\ell$, with $\ell\ge \tfrac1{64D}|S|$.

\begin{claim}\label{claim:smalldegree}
 There is a vertex $u^*$ in $C$ which has at least $\tfrac1{16 (8 ck \log k)^3}|S|$  endpoints of the paths $P_1,\ldots,P_\ell$ in its second neighborhood, where $c>0$ is the constant of \cref{lem:cycle-handles}. 
\end{claim}

\begin{proof}
Assume for the sake of contradiction that each vertex of $C$ has less than $\tfrac1{16 (8 ck \log k)^3}|S|$  endpoints of the paths $P_1,\ldots,P_\ell$ in its second neighborhood. 

\medskip
\begin{figure}[htb]
 \centering
 \includegraphics[scale=0.9]{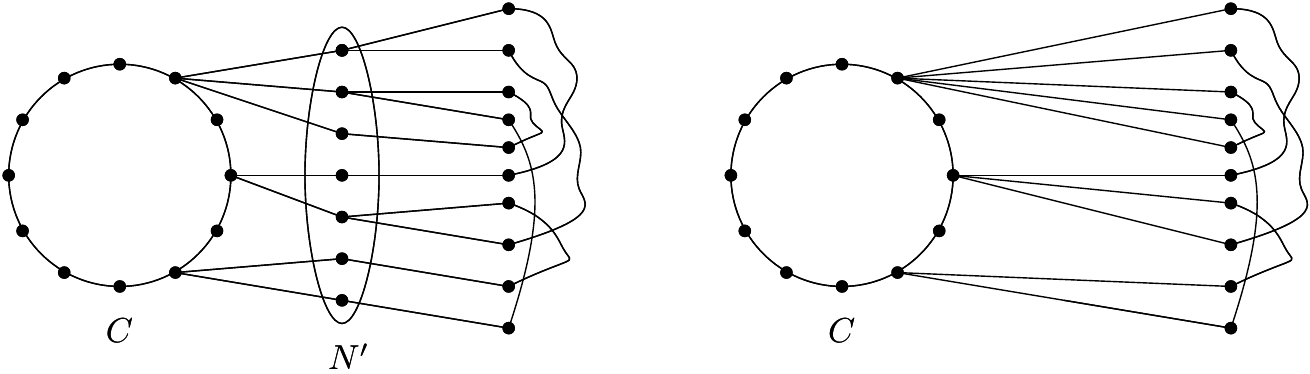}
 \caption{The graphs $G_3$ (left) and $G_4$ (right) in the proof of~\cref{claim:smalldegree}.}
 \label{fig:starpath1}
\end{figure}

Let $G_3$ be subgraph of $G$ induced by $C\cup N'$ and $\bigcup_{i=1}^\ell V(P_i)$ (see Figure \ref{fig:starpath1}, left), and let $G_4$ be the graph obtained from $G_3$ by contracting each vertex of $N'$ with its unique neighbor in $C$ (i.e., $G_4$ is obtained from $G_3$ by contracting disjoint stars into single vertices), see Figure \ref{fig:starpath1}, right. Note that since $G$ is $\mathcal{O}_k$-free, $G_3$ and $G_4$ are also $\mathcal{O}_k$-free (from the structural properties of $C$, $N$, and $S$, each cycle in $G_4$ can be canonically associated to a cycle in $G_3$, and for any set of independent cycles in $G_4$, the corresponding cycles in $G_3$ are also independent). By our assumption, each vertex of $C$ in $G_4$ has degree at most $\tfrac1{16 (8 ck \log k)^3}|S|+2$, and $G_4$ consists of the cycle $C$ together with $\ell\ge \tfrac1{64D}|S|$ paths whose endpoints are in $C$ and whose internal vertices are pairwise disjoint and non-adjacent. By \cref{lem:cycle-handles}, it follows that \[\tfrac1{64D}|S| < \ell\le c \cdot \tfrac1{16 (8 ck \log k)^3}|S|\cdot k \log k,\] and thus $D > 2 (8ck \log k)^2$, which contradicts the definition of $D=2(8ck \log k)^2$.\cqed
\end{proof}

%Old:

% Observe that if the vertices of $N(u^*)$ have average degree at most $(8ck\log k)^2$ in $S'$, then $u^*$ has degree at least $\tfrac1{16(8ck\log k)^5}|S|\ge\delta |S|$, which contradicts Claim \ref{claim:Slinear}. It follows that the vertices of $N(u^*)$ have average degree at least $(8ck\log k)^2$ in $S'$. 

% \medskip

% We now consider the subgraph $G_5$ of $G$ induced by 
% \begin{itemize}[noitemsep,nolistsep]
%     \item the path $C-u^*$
%     \item $N(u^*)$ and the paths $P_i$ ($1\le i \le \ell$) with at least one endpoint in the second neighborhood $N^2(u^*)$ of $u^*$ (call these paths $P_1',\ldots,P_t'$).
%     \item the neighbors of the endpoints of the paths $P_1',\ldots,P_t'$ in $N$.
% \end{itemize}
    
%     \medskip
    
%     In the graph $G_5$, all the components of $G_5-N(u^*)$ are either paths $P_i'$ with both endpoints in $N^2(u^*)$, or a tree, which we denote by $T$ whose leaves are all in $N^2(u^*)$ (except at most two leaves, which are the two neighbors of $u^*$ in $C$). 
    
%     By considering the vertices of $N(u^*)$ and their neighbors in $S'$ as stars (whose centers have average degree at least $(8ck\log k)^2$) we can apply Lemma \ref{lemma:startree_new}, and obtain a vertex of degree at least $\Omega(\tfrac1{k^4 \log k} |S'|)\ge \Omega(\tfrac1{k^6 \log^3 k} |S|)\ge\delta |S|$ in $G_5$ (and thus in $G$), which contradicts Claim \ref{claim:Slinear}. 
    
%     \medskip
    
% This concludes the proof of Theorem~\ref{thm:richvertex}.

%New:
\medskip

 \begin{claim}\label{cl:final}
 If the vertices in $N[u^*]$ have average degree at least $(8ck\log k)^2$ in $S'$, then $G$ contains a vertex of degree at least $\delta |S|$. 
\end{claim} 
\begin{proof}
The key idea of the proof of the claim is to consider the neighbors of $u^*$ as the centers of stars (L) in Claim~\ref{lemma:startree_new}. In order to do that, we consider the subgraph $G_5$ of $G$ induced by 
\begin{itemize}[noitemsep,nolistsep]
    \item the path $C-u^*$,
    \item $N(u^*)$ and the paths $P_i$ ($1\le i \le \ell$) with at least one endpoint in the second neighborhood $N^2(u^*)$ of $u^*$ (call these paths $P_1',\ldots,P_t'$), and
    \item the neighbors of the endpoints of the paths $P_1',\ldots,P_t'$ in $N$.
\end{itemize}
    All the components of $G_5-N(u^*)$ are either paths $P_i'$ with both endpoints in $N^2(u^*)$, or a tree whose leaves are all in $N^2(u^*)$ (except at most two leaves, which are the two neighbors of $u^*$ in $C$). See Figure \ref{fig:starpath2}, right, for an illustration, where the vertices of $N^2(u^*)$ are depicted with squares and the components of $G_5-N(u^*)$ are depicted with bold edges.
    
    \medskip
\begin{figure}[htb]
 \centering
 \includegraphics[scale=0.9]{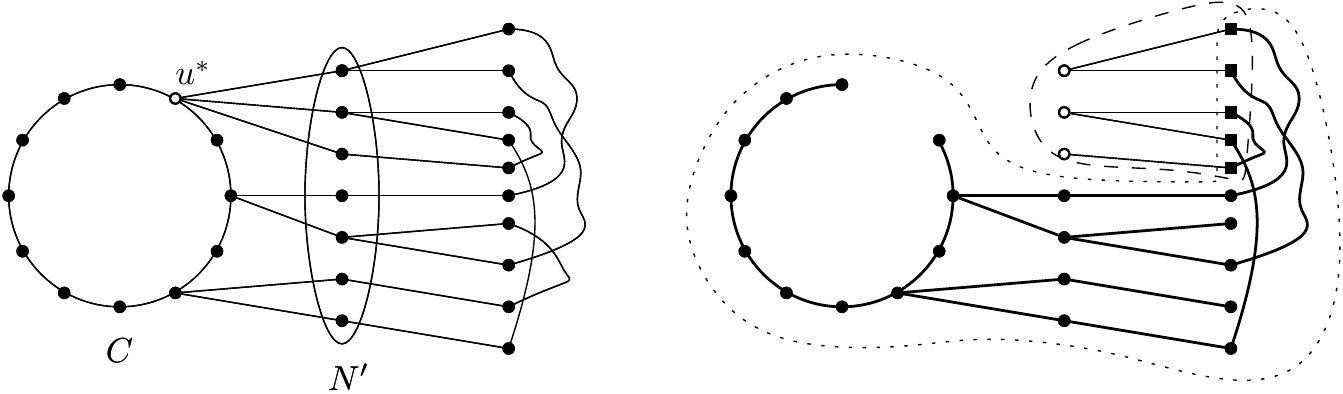}
 \caption{The graphs $G_3$ with the vertex $u^*$ (left) and the graph $G_5$ (right) in the proof of~\cref{cl:final}.}
 \label{fig:starpath2}
\end{figure}
    
    By considering the vertices of $N(u^*)$ and their neighbors in $S'$ as stars (whose centers, depicted in white in Figure \ref{fig:starpath2}, right, have average degree at least $(8ck\log k)^2$) we can apply Lemma \ref{lemma:startree_new}, and obtain a vertex of degree at least $\Omega(\tfrac1{k^4 \log k} |S'|)\ge \Omega(\tfrac1{k^6 \log^3 k} |S|)\ge\delta |S|$ in $G_5$ (and thus in $G$), which contradicts \cref{claim:Slinear}. \cqed
\end{proof}
Observe that if the vertices of $N(u^*)$ have average degree at most $(8ck\log k)^2$ in $S'$, then $u^*$ has degree at least $\tfrac1{16(8ck\log k)^5}|S|\ge\delta |S|$. If not, by \cref{cl:final}, $G$ also contains a vertex of degree at least $\delta |S|$. Both cases contradict \cref{claim:Slinear}, and this concludes the proof of Theorem~\ref{thm:richvertex}.
\hfill $\Box$

%\section{Conclusion}\label{sec:ccl}

\subsection*{Acknowledgements}
 
This work was done during the Graph Theory Workshop held at
Bellairs Research Institute in March 2022. We thank the organizers and the other
workshop participants for creating a productive working
atmosphere. We also thank the anonymous reviewers of the conference and journal versions of the paper for their comments and suggestions.

\bibliographystyle{alpha}
\bibliography{coolnew}

\newcommand{\etalchar}[1]{$^{#1}$}
\begin{thebibliography}{ACHS23b}

\bibitem[AAC{\etalchar{+}}22]{Abrishami22}
Tara Abrishami, Bogdan Alecu, Maria Chudnovsky, Sepehr Hajebi, Sophie Spirkl,
  and Kristina Vu\v{s}kovi\'c.
\newblock Induced subgraphs and tree decompositions {V}. {O}ne neighbor in a
  hole, 2022.

\bibitem[ACHS22]{ACH22}
Tara Abrishami, Maria Chudnovsky, Sepehr Hajebi, and Sophie Spirkl.
\newblock Induced subgraphs and tree-decompositions {III}.
  {T}hree-path-configurations and logarithmic tree-width.
\newblock {\em Advances in Combinatorics}, 2022.

\bibitem[ACHS23a]{ACHS23a}
Bogdan Alecu, Maria Chudnovsky, Sepehr Hajebi, and Sophie Spirkl.
\newblock Induced subgraphs and tree decompositions {IX}. {G}rid theorem for
  perforated graphs, 2023.

\bibitem[ACHS23b]{ACHS23b}
Bogdan Alecu, Maria Chudnovsky, Sepehr Hajebi, and Sophie Spirkl.
\newblock Induced subgraphs and tree decompositions {XII}. {G}rid theorem for
  pinched graphs, 2023.

\bibitem[Ale82]{Alekseev82}
Vladimir~E Alekseev.
\newblock The effect of local constraints on the complexity of determination of
  the graph independence number.
\newblock {\em Combinatorial-algebraic methods in applied mathematics}, pages
  3--13, 1982.

\bibitem[AWZ17]{Artmann17}
Stephan Artmann, Robert Weismantel, and Rico Zenklusen.
\newblock A strongly polynomial algorithm for bimodular integer linear
  programming.
\newblock In Hamed Hatami, Pierre McKenzie, and Valerie King, editors, {\em
  Proceedings of the 49th Annual {ACM} {SIGACT} Symposium on Theory of
  Computing, {STOC} 2017, Montreal, QC, Canada, June 19-23, 2017}, pages
  1206--1219. {ACM}, 2017.

\bibitem[BBB{\etalchar{+}}18]{BonamyBBCT18}
Marthe Bonamy, Edouard Bonnet, Nicolas Bousquet, Pierre Charbit, and
  St{\'{e}}phan Thomass{\'{e}}.
\newblock {EPTAS} for max clique on disks and unit balls.
\newblock In Mikkel Thorup, editor, {\em 59th {IEEE} Annual Symposium on
  Foundations of Computer Science, {FOCS} 2018, Paris, France, October 7-9,
  2018}, pages 568--579. {IEEE} Computer Society, 2018.

\bibitem[BBB{\etalchar{+}}21]{BBB21}
Marthe Bonamy, \'{E}douard Bonnet, Nicolas Bousquet, Pierre Charbit, Panos
  Giannopoulos, Eun~Jung Kim, Pawe\l{} Rz\k{a}\.{z}ewski, Florian Sikora, and
  St\'{e}phan Thomass\'{e}.
\newblock {EPTAS} and subexponential algorithm for maximum clique on disk and
  unit ball graphs.
\newblock {\em J. ACM}, 68(2), jan 2021.

\bibitem[BBD{\etalchar{+}}]{soda}
Marthe Bonamy, Edouard Bonnet, Hugues Déprés, Louis Esperet, Colin Geniet,
  Claire Hilaire, Stéphan Thomassé, and Alexandra Wesolek.
\newblock Sparse graphs with bounded induced cycle packing number have
  logarithmic treewidth.
\newblock In {\em Proceedings of the 2023 Annual ACM-SIAM Symposium on Discrete
  Algorithms (SODA)}, pages 3006--3028.

\bibitem[Ber84]{Bertossi84}
Alan~A Bertossi.
\newblock Dominating sets for split and bipartite graphs.
\newblock {\em Information processing letters}, 19(1):37--40, 1984.

\bibitem[BFMR14]{Bock14}
Adrian Bock, Yuri Faenza, Carsten Moldenhauer, and Andres~J. Ruiz{-}Vargas.
\newblock Solving the stable set problem in terms of the odd cycle packing
  number.
\newblock In Venkatesh Raman and S.~P. Suresh, editors, {\em 34th International
  Conference on Foundation of Software Technology and Theoretical Computer
  Science, {FSTTCS} 2014, December 15-17, 2014, New Delhi, India}, volume~29 of
  {\em LIPIcs}, pages 187--198. Schloss Dagstuhl - Leibniz-Zentrum f{\"{u}}r
  Informatik, 2014.

\bibitem[BGK{\etalchar{+}}18]{BonnetG0RS18}
{\'{E}}douard Bonnet, Panos Giannopoulos, Eun~Jung Kim, Pawe\l{}
  Rz\k{a}\.{z}ewski, and Florian Sikora.
\newblock {QPTAS} and subexponential algorithm for maximum clique on disk
  graphs.
\newblock In Bettina Speckmann and Csaba~D. T{\'{o}}th, editors, {\em 34th
  International Symposium on Computational Geometry, SoCG 2018, June 11-14,
  2018, Budapest, Hungary}, volume~99 of {\em LIPIcs}, pages 12:1--12:15.
  Schloss Dagstuhl - Leibniz-Zentrum f{\"{u}}r Informatik, 2018.

\bibitem[BKTW22]{twin-width1}
{\'{E}}douard Bonnet, Eun~Jung Kim, St{\'{e}}phan Thomass{\'{e}}, and
  R{\'{e}}mi Watrigant.
\newblock Twin-width {I:} tractable {FO} model checking.
\newblock {\em J. {ACM}}, 69(1):3:1--3:46, 2022.

\bibitem[CFH{\etalchar{+}}20]{Conforti20}
Michele Conforti, Samuel Fiorini, Tony Huynh, Gwena{\"{e}}l Joret, and Stefan
  Weltge.
\newblock The stable set problem in graphs with bounded genus and bounded odd
  cycle packing number.
\newblock In Shuchi Chawla, editor, {\em Proceedings of the 2020 {ACM-SIAM}
  Symposium on Discrete Algorithms, {SODA} 2020, Salt Lake City, UT, USA,
  January 5-8, 2020}, pages 2896--2915. {SIAM}, 2020.

\bibitem[CFK{\etalchar{+}}15]{Cygan15}
Marek Cygan, Fedor~V. Fomin, Lukasz Kowalik, Daniel Lokshtanov, D{\'{a}}niel
  Marx, Marcin Pilipczuk, Michal Pilipczuk, and Saket Saurabh.
\newblock {\em Parameterized Algorithms}.
\newblock Springer, 2015.

\bibitem[Cou90]{Courcelle90}
Bruno Courcelle.
\newblock The monadic second-order logic of graphs. {I}. {R}ecognizable sets of
  finite graphs.
\newblock {\em Information and Computation}, 85(1):12--75, 1990.

\bibitem[CP84]{Corneil84}
Derek~G. Corneil and Yehoshua Perl.
\newblock Clustering and domination in perfect graphs.
\newblock {\em Discrete Applied Mathematics}, 9(1):27--39, 1984.

\bibitem[Dav22]{Dav22}
James Davies.
\newblock Oberwolfach report 1/2022.
\newblock doi:10.4171/OWR/2022/1., 2022.

\bibitem[DM{\v{S}}21]{Dallard21}
Cl{\'{e}}ment Dallard, Martin Milani{\v{c}}, and Kenny {\v{S}}torgel.
\newblock Tree decompositions with bounded independence number and their
  algorithmic applications.
\newblock {\em CoRR}, abs/2111.04543, 2021.

\bibitem[DP20]{DP21}
Zden\v{e}k Dvo\v{r}\'ak and Jakub Pek\'arek.
\newblock Induced odd cycle packing number, independent sets, and chromatic
  number, 2020.

\bibitem[Dvo18]{Dvo18}
Zden\v{e}k Dvo\v{r}\'ak.
\newblock Induced subdivisions and bounded expansion.
\newblock {\em European Journal of Combinatorics}, 69:143--148, 2018.

\bibitem[EP65]{EP65}
Paul Erd\H{o}s and Lajos P\'osa.
\newblock On independent circuits contained in a graph.
\newblock {\em Canadian Journal of Mathematics}, 17:347--352, 1965.

\bibitem[FJWY21]{Fiorini21}
Samuel Fiorini, Gwena{\"{e}}l Joret, Stefan Weltge, and Yelena Yuditsky.
\newblock Integer programs with bounded subdeterminants and two nonzeros per
  row.
\newblock In {\em 62nd {IEEE} Annual Symposium on Foundations of Computer
  Science, {FOCS} 2021, Denver, CO, USA, February 7-10, 2022}, pages 13--24.
  {IEEE}, 2021.

\bibitem[GKPT12]{Golovach12}
Petr~A. Golovach, Marcin Kaminski, Dani{\"{e}}l Paulusma, and Dimitrios~M.
  Thilikos.
\newblock Induced packing of odd cycles in planar graphs.
\newblock {\em Theor. Comput. Sci.}, 420:28--35, 2012.

\bibitem[GLP{\etalchar{+}}21]{Gartland21}
Peter Gartland, Daniel Lokshtanov, Marcin Pilipczuk, Michal Pilipczuk, and
  Pawel Rzazewski.
\newblock Finding large induced sparse subgraphs in
  \emph{c\({}_{\mbox{{\textgreater}t}}\)} -free graphs in quasipolynomial time.
\newblock In Samir Khuller and Virginia~Vassilevska Williams, editors, {\em
  {STOC} '21: 53rd Annual {ACM} {SIGACT} Symposium on Theory of Computing,
  Virtual Event, Italy, June 21-25, 2021}, pages 330--341. {ACM}, 2021.

\bibitem[Gol04]{Golumbic04}
Martin~Charles Golumbic.
\newblock {\em Algorithmic graph theory and perfect graphs}.
\newblock Elsevier, 2004.

\bibitem[IPZ01]{Impagliazzo01}
Russell Impagliazzo, Ramamohan Paturi, and Francis Zane.
\newblock Which problems have strongly exponential complexity?
\newblock {\em J. Comput. Syst. Sci.}, 63(4):512--530, 2001.

\bibitem[Jae74]{Jae74}
Fran\c{c}ois Jaeger.
\newblock On vertex-induced forests in cubic graphs.
\newblock In {\em Proceedings of the 5th Southeastern Conference on
  Combinatorics, Graph Theory and Computing}, pages 501--512, 1974.

\bibitem[KKTW01]{Kral01}
Daniel Kr{\'a}l, Jan Kratochv{\'\i}l, Zsolt Tuza, and Gerhard~J Woeginger.
\newblock Complexity of coloring graphs without forbidden induced subgraphs.
\newblock In {\em International Workshop on Graph-Theoretic Concepts in
  Computer Science}, pages 254--262. Springer, 2001.

\bibitem[KO04]{KO04}
Daniela K{\"{u}}hn and Deryk Osthus.
\newblock Induced subdivisions in {$K_{s,s}$}-free graphs of large average
  degree.
\newblock {\em Combinatorica}, 24(2):287--304, 2004.

\bibitem[Kor23]{Korhonen22}
Tuukka Korhonen.
\newblock Grid induced minor theorem for graphs of small degree.
\newblock {\em J. Comb. Theory, Ser. B}, 160:206--214, 2023.

\bibitem[KST54]{KST}
Tam\'as K{\"o}v{\'a}ri, Vera~T. S{\'o}s, and Paul Tur{\'a}n.
\newblock On a problem of {K}. {Zarankiewicz}.
\newblock {\em Colloquium Mathematicum}, 3:50--57, 1954.

\bibitem[Le17]{Khang17}
Ngoc-Khang Le.
\newblock personal communication.
\newblock 2017.

\bibitem[NSS24]{NSS22}
Tung Nguyen, Alex Scott, and Paul Seymour.
\newblock Induced paths in graphs without anticomplete cycles.
\newblock {\em J. Comb. Theory Ser. B}, page 321–339, 2024.

\bibitem[Pil11]{Pilipczuk11}
Micha\l{} Pilipczuk.
\newblock Problems parameterized by treewidth tractable in single exponential
  time: {A} logical approach.
\newblock In Filip Murlak and Piotr Sankowski, editors, {\em Mathematical
  Foundations of Computer Science 2011 - 36th International Symposium, {MFCS}
  2011, Warsaw, Poland, August 22-26, 2011. Proceedings}, volume 6907 of {\em
  Lecture Notes in Computer Science}, pages 520--531. Springer, 2011.

\bibitem[Pol74]{Poljak74}
Svatopluk Poljak.
\newblock A note on stable sets and colorings of graphs.
\newblock {\em Commentationes Mathematicae Universitatis Carolinae},
  15(2):307--309, 1974.

\bibitem[Ray15]{Ray15}
Jean-Florent Raymond.
\newblock Polynomial number of cycles.
\newblock Open problems presented at the Algorithmic Graph Theory on the
  Adriatic Coast workshop, Koper, Slovenia, 2015.

\bibitem[Ree99]{Reed99}
Bruce~A. Reed.
\newblock Mangoes and blueberries.
\newblock {\em Combinatorica}, 19(2):267--296, 1999.

\bibitem[RS86]{RS86}
Neil Robertson and Paul~D. Seymour.
\newblock Graph minors. {V}. {E}xcluding a planar graph.
\newblock {\em Journal of Combinatorial Theory, Series {B}}, 41(1):92--114,
  1986.

\bibitem[RS95]{Robertson95}
Neil Robertson and Paul~D. Seymour.
\newblock Graph minors. {XIII}. {T}he disjoint paths problem.
\newblock {\em Journal of Combinatorial Theory, Series {B}}, 63(1):65--110,
  1995.

\bibitem[ST21]{Sintiari21}
Ni~Luh~Dewi Sintiari and Nicolas Trotignon.
\newblock ({T}heta, triangle)-free and (even hole, ${K}_4$)-free graphs -
  {P}art 1: Layered wheels.
\newblock {\em J. Graph Theory}, 97(4):475--509, 2021.

\bibitem[Tro22]{Tro22}
Nicolas Trotignon.
\newblock Oberwolfach report 1/2022.
\newblock doi:10.4171/OWR/2022/1., 2022.

\bibitem[Tur41]{Tur41}
Paul Tur{\'a}n.
\newblock On an external problem in graph theory.
\newblock {\em Matematikai {\'{e}}s Fizikai Lapok}, 48:436--452, 1941.

\end{thebibliography}

\end{document}